\documentclass[12pt,twoside,a4paper]{article}
\usepackage[T2A]{fontenc}                      
\usepackage[cp1251]{inputenc}           
\usepackage[all]{xy}
\usepackage{amssymb}
\usepackage{cite}
\usepackage{cmap}
\usepackage{latexsym}
\usepackage{enumerate}
\usepackage{amsmath, amsthm, amscd, amsfonts, amssymb, graphicx, colordvi}
\usepackage[left=1.75cm,right=1.75cm,top=2cm,bottom=2cm]{geometry}
\usepackage{indentfirst}
\usepackage{array}
\usepackage[hypertexnames=false, hyperfootnotes=false, bookmarks=false, colorlinks=true, allcolors=blue]{hyperref}
\usepackage{float}
\usepackage{wrapfig}
\usepackage{upref,enumitem}

\usepackage{cite}              

\setlist{leftmargin=7ex}
\setlist[enumerate,1]{label={\normalfont\arabic*.}, ref={\normalfont\arabic*},  itemsep=-2pt, topsep=3pt
}

\usepackage[singlelinecheck=false]{caption}
\newtheorem{cor}{Corollary}
\newtheorem{conjecture}{Conjecture}
\newtheorem{prop}{Proposition}
\newtheorem{rem}{Remark}

\newtheorem{theorem}{Theorem}
\newtheorem{lemen}{Lemma}
\newtheorem{dfn}{Definition}
\DeclareUnicodeCharacter{2116}{\ensuremath{\mathcal{N}}\textsuperscript{\underline{\ensuremath{\circ}}}}

\begin{document}	
\title{On $\overrightarrow{C_{n}}$-irregular oriented graphs}
\date{}
\author{
	Tatiana Dovzhenok\thanks{Research laboratory ``Mathematics of hybrid intelligence systems'',
		Francisk Skorina Gomel State University,  Gomel, 246028 Belarus. E-mail: \texttt{t.dovzhenok@mail.ru}} 
	\and 	
	Ilya Lukashenko\thanks{Research laboratory ``Algebra and geometry of complex systems'',
		Francisk Skorina Gomel State University,  Gomel, 246028 Belarus. E-mails: \texttt{i.d.lukashenko@gmail.com, yahorfiliuta@gmail.com}} 
	\and 	
	Yahor Filiuta\footnotemark[2]
}
\maketitle
	
\begin{abstract}	
	Let $F$ and $G$ be simple finite oriented graphs (without symmetric arcs). A graph $G$ is called $F$-irregular if any two distinct vertices in $G$ belong to a different number of subgraphs of $G$ isomorphic to $F$. In this paper, we investigate the problem of the existence of $\overrightarrow{C_n}$-irregular graphs, where $\overrightarrow{C_n}$ is an oriented cycle of order $n$ (a strongly connected oriented graph that is formed from a simple undirected cycle $C_n$ on $n$ vertices by orienting each of its edges).
	For every integer $n \ge 3$, we prove that there exists an infinite family of $\overrightarrow{C_n}$-irregular graphs. In addition, we show that the order of a non-trivial $\overrightarrow{C_3}$-irregular graph can be any integer not less than $10$ and no others. We also construct $\overrightarrow{C_4}$-irregular graphs of any order at least $7$ and prove that there are no non-trivial $\overrightarrow{C_4}$-irregular graphs of order less than $7$. \smallskip
		
	\end{abstract}
	
	
	\section{Introduction}
	
	In 1987, Chartrand, Holbert, Oellermann and Swart \cite{1} generalized the classical concept of vertex degree and introduced a new class of graphs: $F$\emph{-irregular graphs}.
	\begin{dfn} Let $F$ and $G$ be graphs. The $F$-degree of a vertex $v$ in $G$, denoted as $F\deg_G(v)$, is the number of subgraphs of $G$ that are isomorphic to $F$ and contain $v$. A graph $G$ is called $F$-irregular if the $F$-degrees of all its vertices are pairwise distinct.\end{dfn}

	For example, let $F$ be a complete graph $K_3$ of order $3$. The $K_3$-irregular graph $D_8$ of order~$8$ is depicted in Fig.~\ref{fig1}. Indeed, any vertex $i$ in $D_8$ belongs to exactly $i$ subgraphs of $D_8$ that are isomorphic to $K_3$, i.e., $K_3\deg_{D_8}(i) = i$, and hence the $K_3$-degrees of all vertices in $D_8$ are distinct.
	
	The concept of $F$-irregular graphs arose from an attempt to construct irregular graphs that would be the opposite of regular graphs.
	However, in any non-trivial graph (of order $2$ or greater), there are two vertices with the same degree (see, for example, \cite{2}). This has led to new approaches to defining irregular graphs, a comprehensive overview of which is presented in the book ``Irregularity in Graphs'' \cite{3}.
	\begin{figure}[h!]
		\centering\includegraphics[width=0.32\textwidth]{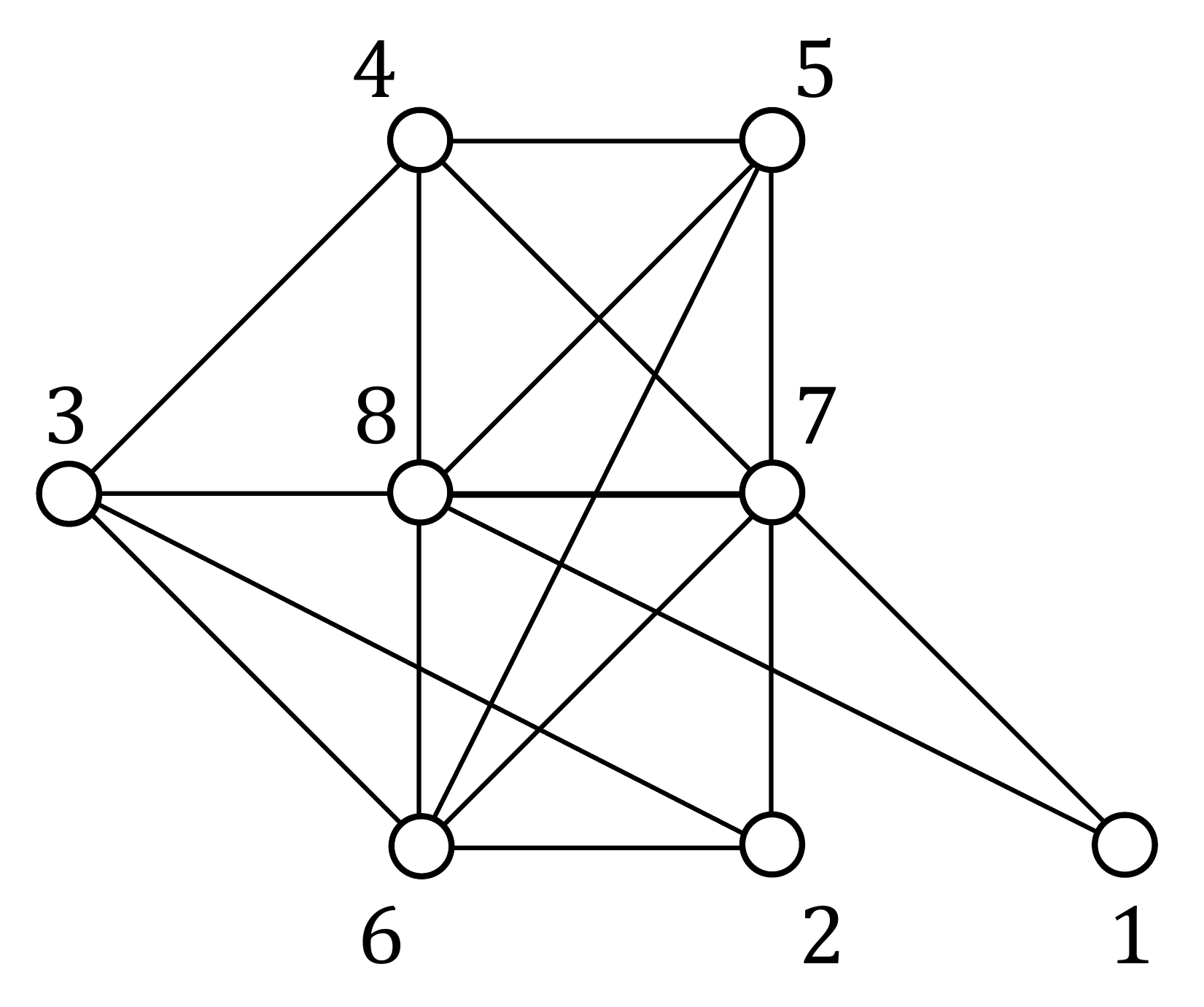}	
		\caption{$K_3$-irregular graph {$D_8$}.}
		\label{fig1}
	\end{figure}
	
	In the field of $F$-irregular graphs, for each graph $F$, the questions of existence, number and order of $F$-irregular graphs are of great   interest. It is known \cite{4} that for the path $P_3$ on three vertices, there exists a $P_3$-irregular graph of any order starting from $6$. $K_3$-irregular graphs of any order starting from $7$ were constructed in \cite{5}.
	
	In  \cite{1}, the existence of non-trivial $F$-irregular graphs was shown in the case where $F$ is a star or a complete graph of order 3 or more, and the following conjecture was proposed.
	\begin{conjecture}\phantomsection\label{con1} \mbox{For every connected graph $F$ of order $3$} or more, there exists a non-trivial
		$F$-ir\-re\-gular graph.
	\end{conjecture}
	
	In 2024, Dovzhenok, Filuta, and Chuhai \cite{6} proved that for any biconnected graph $F$ with  minimum vertex degree  $2$  there exist  infinitely many   $F$-irregular graphs.
	In particular, this result holds for any  simple cycle  $C_n$ on $n \ge 3$ vertices. Thus, the problem of $F$-irregular graphs was first considered for a graph $F$ of arbitrarily large diameter. Additionally,  a more general hypothesis  than Conjecture~\ref{con1} was proposed in \cite{6}.
	\begin{conjecture}[Strong conjecture about $F$-irregular graphs]\phantomsection\label{con2}  For any connected graph $F$ of~or\-der~$|F| \ge 3$,  there exist infinitely many $F$-irregular graphs.
	\end{conjecture}
	
	Recently, Conjecture~\ref{con1} was verified for the $n$-vertex path $P_n$ ($n \geq 3$), while Conjecture~\ref{con2} was confirmed for the path $P_4$~\cite{7}.
	
	
	In this paper, we  consider the problem of $F$-irregular graphs in the class of oriented graphs.
	By an oriented graph, we mean a directed graph without symmetric arcs. In other words, in such a digraph,  for any two of its vertices $u$ and $v$, there are no arcs $(u,v)$ and $(v,u)$ at the same time (here and below, for vertices $u, v$, the arc directed from $u$ to $v$  will be denoted by  $(u, v)$). We will consider only finite oriented graphs without loops and multiple arcs. An oriented graph is called \emph{strongly connected} (\emph{weakly connected}) if for any two of its vertices $u$ and $v$, there exists a directed (undirected) path from $u$ to $v$.
	
	Let $G$ be an oriented graph. We write the set of  vertices of $G$ as $V(G)$, and the set of arcs of $G$ as $A(G)$. The number of vertices in $G$ (its order) is denoted by $|G|$. Also, for any set $X$, we use a similar notation, $|X|$, for the number of elements in $X$. We say that a vertex $v$ is incident to an arc $(a,b) \in A(G)$ if $v \in \{a, b\}$.  The \emph{indegree} and  \emph{outdegree}  of a vertex $v$ in a graph $G$ are defined as
	$$\deg_G^-(v)=|\{u \in V(G) \mid (u,v) \in A(G)\}|, \quad
	\deg_G^+(v)=|\{u \in V(G) \mid (v,u) \in A(G)\}|,$$		
	respectively.
	
	With regard to the question of the existence of $F$-irregular oriented graphs, we believe that a statement similar to Conjecture~\ref{con2} for undirected graphs is true.
	\begin{conjecture}[Strong conjecture about $F$-irregular oriented graphs]\phantomsection\label{con3} For every weakly~con\-nec\-ted orien\-ted graph $F$ of order $|F| \ge 3$, there are infinitely many $F$-irregular oriented graphs.
	\end{conjecture}
	
	In this paper, we confirm Conjecture~\ref{con3} for each oriented cycle.
	\begin{dfn} Let $n \ge 3$ be an integer. An oriented cycle of order $n$, denoted by $\overrightarrow{C_n}$, is a strongly connected oriented graph in which the indegree and outdegree of all vertices are equal to $1$. An oriented  cycle with the vertex set  $\{v_1,v_2,\ldots,v_n\}$ and the arc set  $\{(v_1,v_2),(v_2,v_3),\ldots,(v_{n-1},v_n)$,  $(v_n,v_1)\}$ will be written as $(v_1,v_2,\ldots,v_n)$.\end{dfn}

	In addition, we also determine all the possible values for the order of a $\overrightarrow{C_n}$-irregular graph when $n \in \{3,4\}$.
	
	The structure of the article is as follows. In Section~\ref{dsec2}, for each integer $n \ge 5$, we construct
	an infinite family of $\overrightarrow{C _n }$-irregular graphs. In Section~\ref{dsec3}, we provide  examples of $\overrightarrow{C _4 }$-irregular graphs of every  order $k \ge 7$ and prove that there is no  non-trivial $\overrightarrow{C _4 }$-irregular graph of smaller  order. In Section~\ref{dsec4}, we show that the minimal order of a non-trivial $\overrightarrow{C _3}$-irregular graph is $10$, and we present constructions of $\overrightarrow{C _3 }$-irregular graphs of every  order starting from $10$. Finally, in Section~\ref{dsec5}, we state the main result.
	
	Note that in this paper, we denote the number of $k$-element subsets of an $n$-element set by $C_n^k$, that is, for integers $n \ge k \ge 0$, $C_n ^k= \dfrac {n!}{k!(n-k)!}$, where $0!= 1$, $m!= 1 \cdot 2 \cdot 3 \cdot \ldots \cdot m$ if $m$ is a positive integer, and $C _n ^k=0$ otherwise.
	
	\section{\texorpdfstring{$\overrightarrow{C _n }$}{C_n}-irregular graphs in case $n \ge 5$}\label{dsec2}
	
	\begin{dfn} Let $l, n$ be integers such that $l \ge n \ge 5$. Consider the graph $\overrightarrow{A_{2l+2, n}}$ with the set of vertices
		\[V(\overrightarrow{A_{2l+2, n}})=V_1 \cup V_2 \cup \{2l+1,2l+2\}, \ V_1= \{1,2,\dots,l\}, \ V_2= \{l+1,l+2,\dots,2l\},\] 	
		and the set of arcs
		\begin{align*}
			A(\overrightarrow{A_{2l+2,n}}) &= \{(i,j) \mid i,j \in V_1,i<j\} \cup \{(i,j) \mid i,j \in V_2,i<j\} \cup \{(2l+1,i) \mid i \in V_2\} \\[1ex]
			&\phantom{=} \cup \{(i,j) \mid i \in V_2,j \in V_1,i-j \le l\} \\[1ex]
			&\phantom{=} \cup \{(l,2l+1),(l+n-2,2l+2),(2l+2,2l+1)\}.
	\end{align*}\end{dfn}
	
	For the sake of clarity, we place the vertices of the graph $\overrightarrow{A_{2l+2,\, n}}$ on 3 levels as shown in Fig.~\ref{fig2}. In this case, the vertices of each level will be numbered in ascending order from left to right and so that each vertex $i\in V_1$ is located directly above the vertex $l+i$. Then, from each vertex of the top and middle levels, there is an arc to any other vertex on the same level located to the right of it. Furthermore, from any vertex $i$ of the middle level, there is an arc to every vertex of the top level located directly above it (vertex $i - l$) or to the right of $i - l$. From vertex $l$, there is exactly one arc $(l,2l+1)$. Finally, from vertex $2l+1$, there is an arc to every vertex of the middle level, and vertex $2l+2$ is incident to exactly two arcs $(l+n-2,2l+2)$, $(2l+2,2l+1)$. \smallskip
	
	\begin{dfn} Let $a_i = \overrightarrow{C_n}\deg_{\overrightarrow{A_{2l+2,\, n}}}(i)$ for each $i \in V(\overrightarrow{A_{2l+2,\, n}})$.\end{dfn}
	\begin{lemen}\phantomsection\label{lemma1}
		Let $l, n$ be integers such that $l \ge n \ge 5$. Then the $\overrightarrow{C_n}$-degrees of the vertices in the graph $\overrightarrow{A_{2l+2,\, n}}$ are equal to\smallskip

		$1) \ a_i=iC_{l-2}^{n-4} \ \  \forall i \in \{1,2,\ldots,l-1\};$\smallskip

		$2) \ a_l=lC_{l-1}^{n-3};$\smallskip

		$3) \ a_i=(2l-i)C_{l-2}^{n-4}+C_{l-1}^{n-3}+1 \ \  \forall i \in \{l+1,l+2,\ldots,l+n-2\};$\smallskip

		$4) \ a_i=(2l-i)C_{l-2}^{n-4}+C_{l-1}^{n-3} \ \   \forall i \in \{l+n-1, l+n,\ldots,2l\};$\smallskip

		$5) \ a_{2l+1}=lC_{l-1}^{n-3}+1;$\smallskip

		$6) \ a_{2l+2}=1.$
	\end{lemen}
	\begin{figure}[t]
		\centering
		\includegraphics[width=0.6\textwidth]{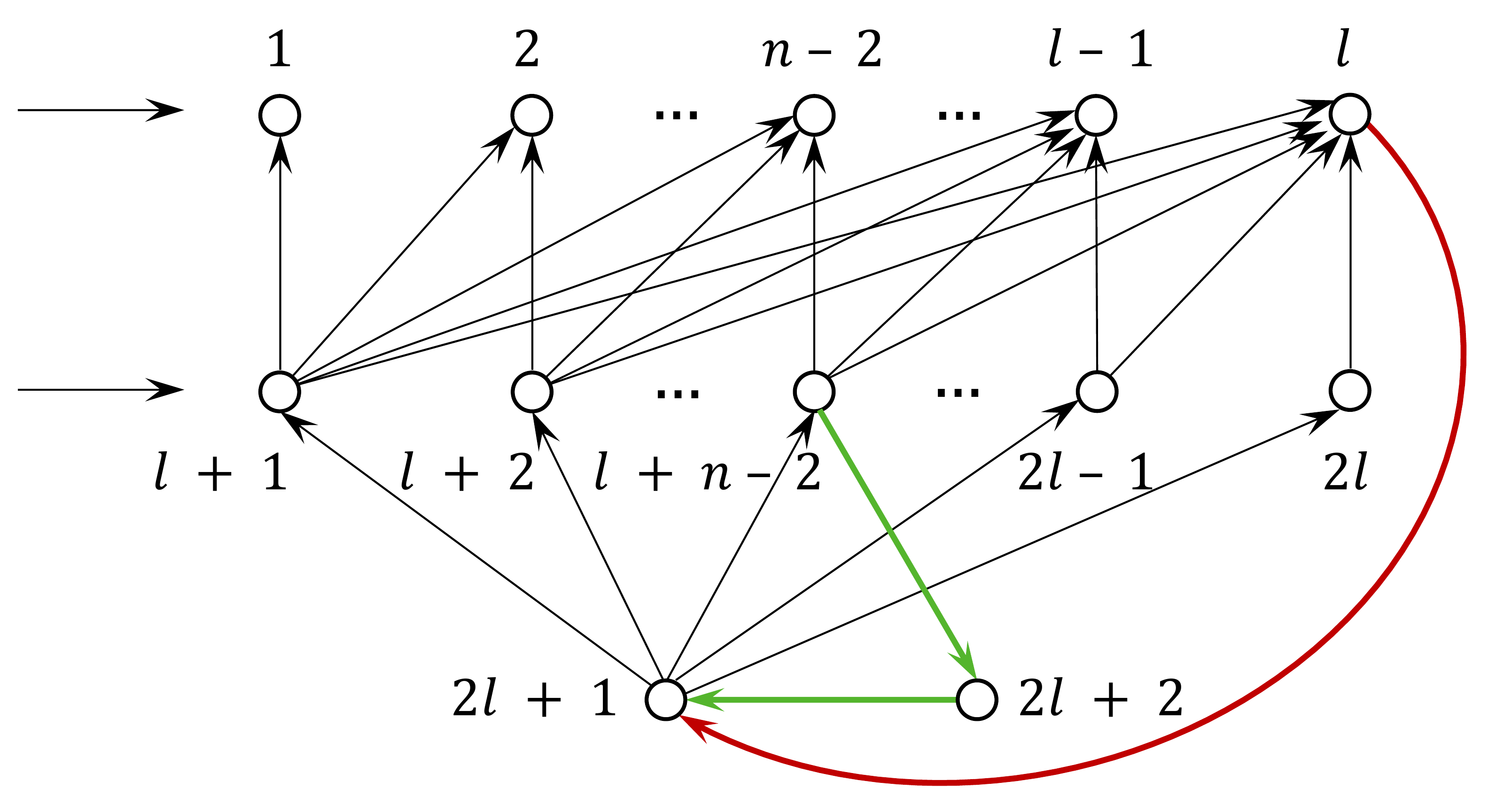}
		\caption{Graph $\overrightarrow{A_{2l+2,\, n}}$.}
		\label{fig2}
	\end{figure}
	
	\begin{proof}
		In the proof of Lemma~\ref{lemma1} and also in  Corollary~\ref{cor1}, any subgraph of the graph $\overrightarrow{A_{2l+2,\, n}}$ that is isomorphic to $\overrightarrow{C_n}$ will be called a \emph{contour}.	
		Let's take a walk along the arcs of each contour, following their direction.
		Based on the structure of  $\overrightarrow{A_{2l+2,\, n}}$ (see Fig.~\ref{fig2}), the traversal of each contour is subject to the following rules:
		\begin{enumerate}[label=R\arabic*., noitemsep, topsep=4pt]
			
			\item On the top and middle levels of the graph $\overrightarrow{A_{2l+2,\, n}}$, movement is only possible from left to right.
				
			\item  Descent from the top level is only possible via the arc $(l,2l+1)$.
			
			\item From any vertex $i$ of the middle level, one can go  along an arc to a vertex $j$ of the top level if and only if $i-l \leq j \leq l$.
			
			\item Descent from the middle level to the bottom level is only possible via the arc $(l+n-2,2l+2)$.
			
			\item From vertex $2l+1$, you can go along the arc to any vertex of the middle level, and only to those vertices. 
		\end{enumerate}
		\begin{prop}\phantomsection\label{p1}
			There is exactly one contour containing vertex $2l+2$\textup{:}
			$$(l+1,l+2,\ldots,l+n-3,l+n-2,2l+2,2l+1).$$	
		\end{prop}
		\begin{proof}
			Based on the definition of $\overrightarrow{A_{2l+2,\, n}}$, every contour with  vertex $2l+2$  must include the arcs $(l+n-2,2l+2)$, $(2l+2,2l+1)$, and therefore cannot contain the arc $(l,2l+1)$. Hence, taking into account R2, such contours do not contain any top-level vertices. Then, by  R1, R4, they do not contain any vertices from the set $\{l+n-1,l+n,\ldots,2l\}$. Thus, the vertices of each contour containing $2l+2$ belong to the set $V=\{l+1,l+2,\ldots,l+n-2\} \cup \{2l+1,2l+2\}.$ Clearly,  there is only one desired contour with vertices from $V$: $(l+1,l+2,\ldots,l+n-2,2l+2,2l+1)$. 	
		\end{proof}
	
		By Proposition {\ref{p1}}, $a_{2l+2}=1$. Further,
		any contour without vertex $2l+2$ will be called a \emph{base  contour}. From R1, R2, R4, R5 it follows that every base  contour contains vertices $l,2l+1$ and at least one vertex at the middle level of $\overrightarrow{A_{2l+2,\, n}}$.
		
		Fix $i \in V(\overrightarrow{A_{2l+2,\, n}}) \backslash \{2l+2\}$. Let us call an $[i,j]$-\emph{contour} any base contour with vertices $i,j$, where $j$ is the rightmost vertex of this contour in the middle level of $\overrightarrow{A_{2l+2,\, n}}$.
		It is easy to see that the number of all base contours with vertex $i$ is equal to the number of all $[i,j]$-contours, where $j \in \{l+1,l+2,\ldots,2l\}$. Next, we say that an ordered pair of vertices $[i,j]$ in $\overrightarrow{A_{2l+2,n}}$ is \emph{correct} if the following two conditions are satisfied:
		\[
		j \in \{l+1, l+2,\dots, 2l\} \quad \text{and} \quad i=2l+1 \text{ or } 0 \leq j-i \leq l,
		\]
		and \emph{incorrect} otherwise.
		
		From R1--R4 and the definition of an $[i,j]$-contour, one can conclude that for any incorrect pair $[i,j]$ there is no $[i,j]$-contour.		
		\begin{prop}\phantomsection\label{p2}
			For every correct pair $[i,j]$, the number of $[i,j]$-contours  is equal to $C_{l-1}^{n-3}$ if $i \in \{l, 2l+1, j\}$, and $C_{l-2}^{n-4}$ otherwise.
		\end{prop}
		\begin{proof}
			Let $[i,j]$ be a correct pair. Based on R1-R4 and the definition of an $[i,j]$-contour, we conclude that the vertices of any $[i,j]$-contour belong to the following set: $$V_j=\{j-l, j-l+1, \ldots, j\} \cup \{2l+1\}.$$
			
			Now, let's consider two cases for $i$.
			\smallskip
			
			\textbf{Case 1.} $i \in \{l, 2l+1, j\}$. First of all, we note that  any $[l,j]$-contour has exactly $3$ fixed vertices $l, j, 2l+1$, and the remaining  $n-3$ vertices for such contours  can be chosen in $C_{l-1}^{n-3}$ ways from the set $V_j \backslash \{l, j, 2l+1\}$. Next, we show that from any $n-3$  selected vertices, together with the vertices $l, j, 2l+1$, it is possible to construct exactly one $[l,j]$-contour, unique for each set of $n-3$ selected vertices. It is indeed true that if for some positive integer $ x \leq n-2$, the vertices $u_1<u_2<\ldots<u_x=l$ belong to the set $\{j-l, j-l+1, \ldots, l\}$, and the vertices $v_1<v_2<\ldots<v_{n-1-x}=j$ belong to the set $\{l+1, l+2, \ldots, j\}$, then, taking into account R1-R5, we can construct only one unique $[l,j]$-contour with vertices $u_1, u_2, \ldots, u_x, v_1, v_2, \ldots, v_{n-x-1}, 2l+1$:
			$(u_1, u_2,\ldots, u_x, 2l+1, v_1, v_2,  \ldots,v_{n-x-1}).$	
			Hence, the number of $[l,j]$-contours is equal to $C_{l-1}^{n-3}$.
			It remains to be noted that every $[j,j]$-contour and every $[2l+1,j]$-contour is also an $[l,j]$-contour, and vice versa. Therefore, the number of $[j,j]$-contours, as well as the number of $[2l+1,j]$-contours, is equal to $C_{l-1}^{n-3}$.
			\smallskip
			
			\textbf{Case 2.} $i \notin \{l, 2l+1, j\}$. In this case, any $[i,j]$-contour has exactly $4$ fixed vertices $i, j, l, 2l+1$, and the remaining $n-4$ vertices for such contours can be chosen in $C_{l-2}^{n-4}$ ways from the set $V_j \backslash \{i, j, l, 2l+1\}$. Similarly to case 1, it can be proven that from any $n-4$ selected vertices and the vertices $i, j, l, 2l+1$, one can form exactly one $[i,j]$-contour, with different selections corresponding to different $[i,j]$-contours. Thus, the number of $[i,j]$-contours is $C_{l-2}^{n-4}$ if $i \notin \{l, 2l+1, j\}$.		
		\end{proof} 	
		\medskip
		
		Let us assume that $i \in \{1,2,\ldots,l-1\}$. By definition, $[i,j]$ is a correct pair if and only if $j \in \{l+1,l+2,\ldots,l+i\}$. Thus, taking into account Propositions {\ref{p1}}, {\ref{p2}}, we find $a_i$: 		
		$$a_i=\sum _ {j=l+1}^{l+i} C_{l-2}^{n-4}=iC_{l-2}^{n-4}\ \  \forall i \in \{1,2,\ldots,l-1\}.$$ 			
		
		If $i \in \{l+1,l+2,\ldots,2l\}$, then a pair $[i,j]$ will be correct if and only if $j \in \{i,i+1,\ldots,2l\}$. Hence, by Propositions {\ref{p1}}, {\ref{p2}}, we calculate $a_i$:		
		\begin{equation*}\begin{split}a_i&=C_{l-1}^{n-3}+\sum_{j=i+1}^{2l}C_{l-2}^{n-4}+1= (2l-i)C_{l-2}^{n-4}+C_{l-1}^{n-3}+1 \ \  \forall i \in \{l+1,l+2,\ldots,l+n-2\},	\\		
				a_i&=C_{l-1}^{n-3}+\sum_{j=i+1}^{2l}C_{l-2}^{n-4}= (2l-i)C_{l-2}^{n-4}+C_{l-1}^{n-3} \ \  \forall i \in \{l+n-1,l+n,\ldots,2l\}.\end{split}\end{equation*}		
		
		To find $a_l$, we notice that a pair $[l,j]$ is correct if and only if  $j \in \{l+1,l+2,\ldots,2l\}$. Then, based on Propositions {\ref{p1}}, {\ref{p2}}, we get
		$$a_l=\sum_{j=l+1}^{2l}C_{l-1}^{n-3}= lC_{l-1}^{n-3}.$$		
		
		Finally, we compute $a_{2l+1}$. Since each base contour has vertices $l, 2l+1$, and by Proposition {\ref{p1}}, the contour with vertex $2l+2$ contains vertex $2l+1$ and does not contain vertex $l$, then
		\[a_{2l+1}=a_l+1= lC_{l-1}^{n-3}+1. \qedhere\]	
	\end{proof}
	\begin{cor}\phantomsection\label{cor1}
		Let $l \ge n \ge 5$ be integers. Then the following inequalities are true:
		
		\smallskip 	
		$1) \ 1<a_1<a_2<a_3<\ldots<a_{l-1};$
		
		\smallskip	
		$2) \ a_{l+1}>a_{l+2}>a_{l+3}>\ldots>a_{2l}>1;$
		
		\smallskip
		$3) \ a_{2l+1}>a_l>a_{l+1}>a_{l-1}.$  	
	\end{cor}
	\begin{proof}
		Inequalities 1), 2), as well as the inequalities  $a_{2l+1}>a_l$ and $a_{l+1}>a_{l-1}$, follow directly from Lemma~\ref{lemma1}. 	
		Let's show that $a_l>a_{l+1}$.
		By the proof of Lemma~\ref{lemma1}, every contour in $\overrightarrow{A_{2l+2,\, n}}$  except one contains vertex $l$. On the other hand, two contours   $(l,2l+1,l+2,l+3,\ldots,l+n-1)$ and $(l,2l+1,l+3,l+4,\ldots,l+n)$ do not contain  vertex $l+1$. Hence, $a_l>a_{l+1}$.	
	\end{proof}
	\begin{lemen}\phantomsection\label{lemma2}
		Let $l \ge n \ge 5$ be integers and  $l-1$ is not divisible by $n-3$. Then
		$a_i \ne a_j$  for all $i \in V_1 \backslash \{l\}, \ j \in V_2.$
	\end{lemen}
	\begin{proof}
		Suppose, for the contrary, that  $a_i=a_j$ for some $i \in V_1 \backslash \{l\}, \ j \in V_2$. 	
		We distinguish two cases for $j$.
		\smallskip
		
		\textbf{Case 1.} Let $j \in \{l+n-1,l+n,\ldots,2l\}$. By Lemma~\ref{lemma1}, for integers $l \ge n \ge 5$, we have
		\begin{multline*} a_i=a_j      \iff  iC_{l-2}^{n-4}=(2l-j)C_{l-2}^{n-4}+C_{l-1}^{n-3}\\
			\iff  (i+j-2l)\frac{(l-2)!}{(n-4)!(l-n+2)!}=\frac{(l-1)!}{(n-3)!(l-n+2)!}\\
			\iff  (i+j-2l)(n-3)=l-1.\end{multline*}
		It follows from the last equality that $l-1$ is divisible by $n-3$, which contradicts the condition of Lemma~\ref{lemma2}.   From this we conclude that
		$a_i \ne a_j$ for all $i\in V_1 \backslash \{l\}$, $j \in \{l+n-1,l+n,\ldots,2l\}.$
		\smallskip
		
		\textbf{Case 2.} Let $j \in \{l+1,l+2,\ldots,l+n-2\}$. Based on Lemma~\ref{lemma1}, for integers $l \ge n \ge 5$, we can equivalently transform the equality $a_i=a_j$ as follows: 	
		\begin{multline*}a_i=a_j      \iff iC_{l-2}^{n-4}=(2l-j)C_{l-2}^{n-4}+C_{l-1}^{n-3}+1\\
			\iff  (i+j-2l)\frac{(l-2)!}{(n-4)!(l-n+2)!}=\frac{(l-1)!}{(n-3)!(l-n+2)!}+1\\
			\iff  (i+j-2l)(n-3)-l+1=\frac{(n-3)!(l-n+2)!}{(l-2)!}.\end{multline*}
		From the last equality we obtain that the fraction $\dfrac{(n-3)!(l-n+2)!}{(l-2)!}$ must be an integer for integers $l \ge n \ge 5$. However, this is not true since for given $l$ and $n$, the following estimate holds:
		\begin{equation*}0<\frac{(n-3)!(l-n+2)!}{(l-2)!}=\frac{(n-3)!}{(l-n+3)(l-n+4)\ldots(l-2)}
			\leq \frac{(n-3)!}{3\cdot4\cdot \ldots\cdot (n-2)}=\frac{2}{n-2}<1.\end{equation*}	
		Consequently,
		$a_i \ne a_j$ for all $i \in V_1 \backslash \{l\}$, $j \in \{l+1,l+2,…,l+n-2\}.$
	\end{proof}
	\begin{theorem}\phantomsection\label{t1}
		For any integers $l, n$ with $l \ge n \ge 5$ and $l-1$ is not divisible by $n-3$, the graph $\overrightarrow{A_{2l+2,\, n}}$ is $\overrightarrow{C _n }$-irregular.
	\end{theorem}
	\begin{proof}
		Let $l \ge n \ge 5$ be integers and $l-1$ is not divisible by $n-3$. Let us prove that the graph $\overrightarrow{A_{2l+2,\, n}}$ is $\overrightarrow{C _n }$-irregular.
		First of all, we note that by Corollary~\ref{cor1}, in each of the sets $V_1 \backslash \{l\}$ and $V_2$, the vertices have different $\overrightarrow{C _n}$-degrees. Furthermore, based on Lemma~\ref{lemma1} and Corollary~\ref{cor1}, we conclude that in the graph~$\overrightarrow{A_{2l+2,\, n}}$ vertex $2l+1$ has the largest $\overrightarrow{C _n}$-degree, vertex $l$ has the second largest, and vertex $2l+2$ has the smallest $\overrightarrow{C _n }$-degree.
		Finally, by Lemma~\ref{lemma2}, any two vertices, one of which belongs to the set $V_1 \backslash \{l\}$ and the other to the set $V_2$, also have different $\overrightarrow{C _n}$-degrees in  $\overrightarrow{A_{2l+2,\, n}}$.
		Thus, the $\overrightarrow{C _n}$-degrees of all vertices in $\overrightarrow{A_{2l+2,\, n}}$ are pairwise distinct.
		Hence, the graph $\overrightarrow{A_{2l+2,\, n}}$ is $\overrightarrow{C _n}$-irregular.
	\end{proof}
	\begin{cor}\phantomsection\label{cor2}
		For each integer $n\ge5$, there are infinitely many $\overrightarrow{C _n}$-irregular oriented graphs.
	\end{cor}
	
	\section{On the order of \texorpdfstring{$\overrightarrow{C _4 }$}{C_4}-irregular graphs}\label{dsec3}
	In this section, we focus our efforts on determining all values of $k$ for which there exists a $\overrightarrow{C_4}$-irregular graph of order $k$.
	\begin{dfn}In our work, every graph that is isomorphic to $\overrightarrow{C _4 }$ will be called a quadrangle.
	\end{dfn}
	
	\subsection{\texorpdfstring{$\overrightarrow{C_4}$}{C_4}-irregular graph of order 7}
	
	\begin{dfn}Let's define the graph $\overrightarrow{B_7}$ (see Fig.~\ref{fig3}) as follows:	
		\begin{align*}
			V(\overrightarrow{B_7})=\{1,2,\ldots,7\}, \quad
			A(\overrightarrow{B_7})&=\{(1,6),(2,6),(3,7),(4,2),(4,3),(4,6),(5,1),(5,2)\}\\
			&\phantom{=} \cup \{(5,3),(5,4),(5,6),(6,3),(6,7),(7,4),(7,5)\}. 	
	\end{align*} \end{dfn}	
\vspace{-1ex}
	\begin{figure}[h!]
		\centerline{\includegraphics[width=0.35\textwidth]{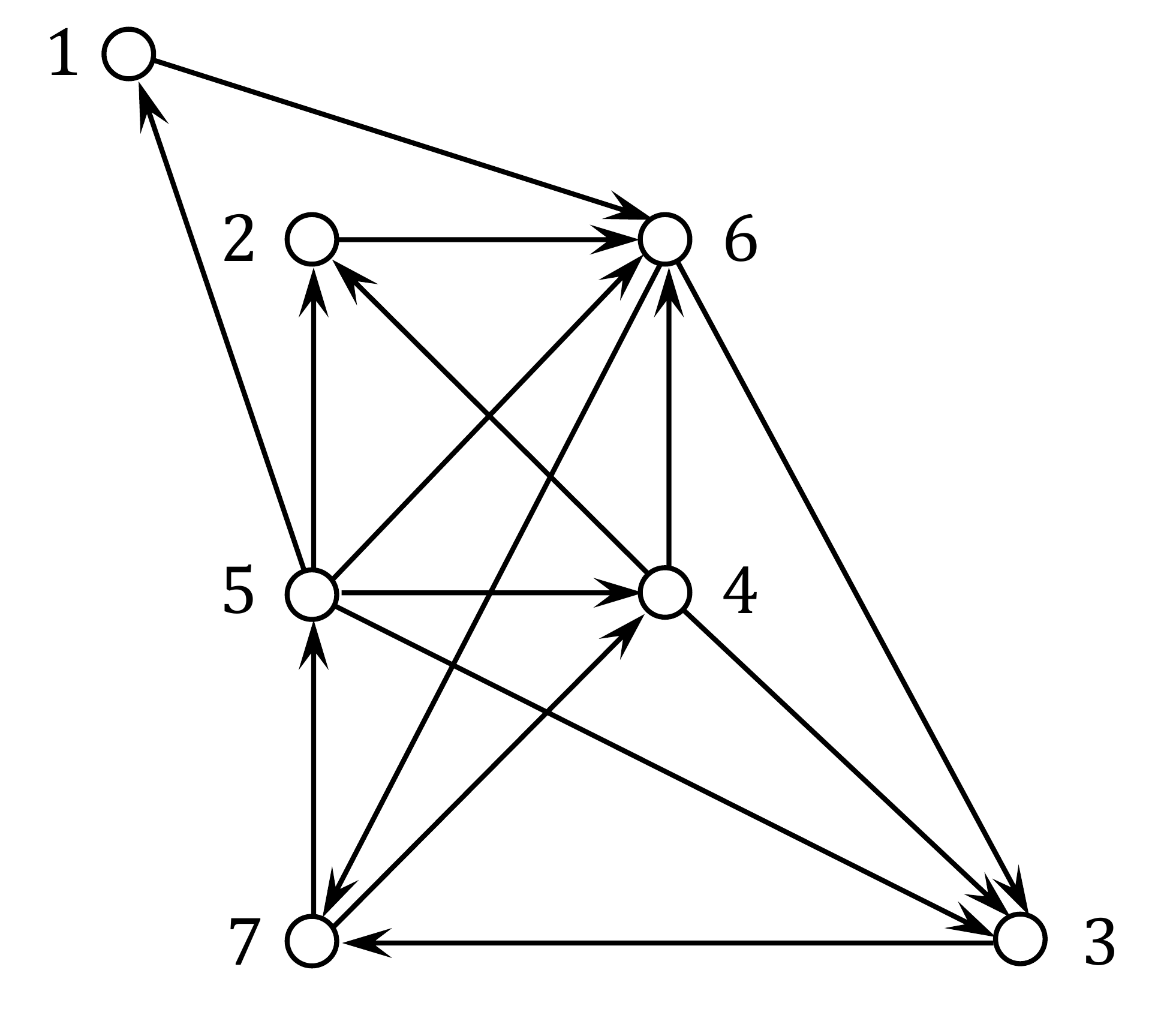}}
		\caption{Graph $\overrightarrow{B_7}$.}
		\label{fig3}
	\end{figure}	
\vspace{-1ex}
	\begin{lemen}\phantomsection\label{lemma3}
		The graph $\overrightarrow{B_7}$ is $\overrightarrow{C _4 }$-irregular.	
	\end{lemen}
	\begin{proof}
		It is easy to see that the graph	 $\overrightarrow{B_7}$ contains exactly $7$ quadrangles: $(1,6,7,5)$, $(2,6,7,4), (2,6,7,5), (3,7,4,6), (3,7,5,4), (3,7,5,6), (4,6,7,5).$ The $\overrightarrow{C _4 }$-degrees of all vertices in $\overrightarrow{B_7}$ are listed in Table~\ref{dtabl1}. Since they are distinct, $\overrightarrow{B_7}$ is a $\overrightarrow{C _4 }$-irregular graph.
	\end{proof}
	\begin{table} [b!]	\small
		\caption{$\overrightarrow{C_4 }$-degrees of vertices in graph $\overrightarrow{B_7}$}\label{dtabl1}
		\centerline{\begin{tabular}{ | m{4.5em} | m{0.4cm}| m{0.4cm} |m{0.4cm} |m{0.4cm} |m{0.4cm} |m{0.4cm} |m{0.4cm} | }
				\hline
				vertex & 1 & 2 & 3 & 4 & 5 & 6 & 7 \\
				\hline
				
				$\overrightarrow{C_4}$-degree & 1 & 2 & 3 & 4 & 5 & 6 & 7 \\
				\hline
		\end{tabular}}
	\end{table}	
	
	\subsection{\texorpdfstring{$\overrightarrow{C_4}$}{C_4}-irregular graph of order 8}
	
	\begin{dfn} Consider the graph $\overrightarrow{B_8}$ (see Fig.~\ref{fig4}) with the set of vertices $V(\overrightarrow{B_8})=\{1,2,\dots,8\}$ and the set of arcs
		\begin{align*}
			A(\overrightarrow{B_8})&=\{(1,2),(1,3),(2,3),(3,7),(3,8),(4,1),(4,2),(4,3),(4,5),(4,6),(4,8),(5,2)\}\\
			&\phantom{=} \cup \{(5,3),(5,6),(5,8),(6,3), (6,8),(7,2),(7,4),(7,5),(7,6),(8,1),(8,7)\}.
	\end{align*}\end{dfn}
	\begin{rem}\label{rem1}
		Here and below, when depicting graphs, an arc with a direction from the set of vertices $X$ to the set of vertices $Y$ will denote the set of arcs $\{(x,y) \mid x\in X, \ y \in Y\}$. For example, in Fig.~\ref{fig4},  the arc with direction from  vertex $7$ to the set of vertices $\{4, 5, 6\}$ implies that the graph~$\overrightarrow{B_8}$ contains the arcs $(7, 4)$, $(7, 5)$, $(7, 6)$.\end{rem}
	\begin{figure}[h!]\small
		\centerline{\includegraphics[width=0.4\textwidth]{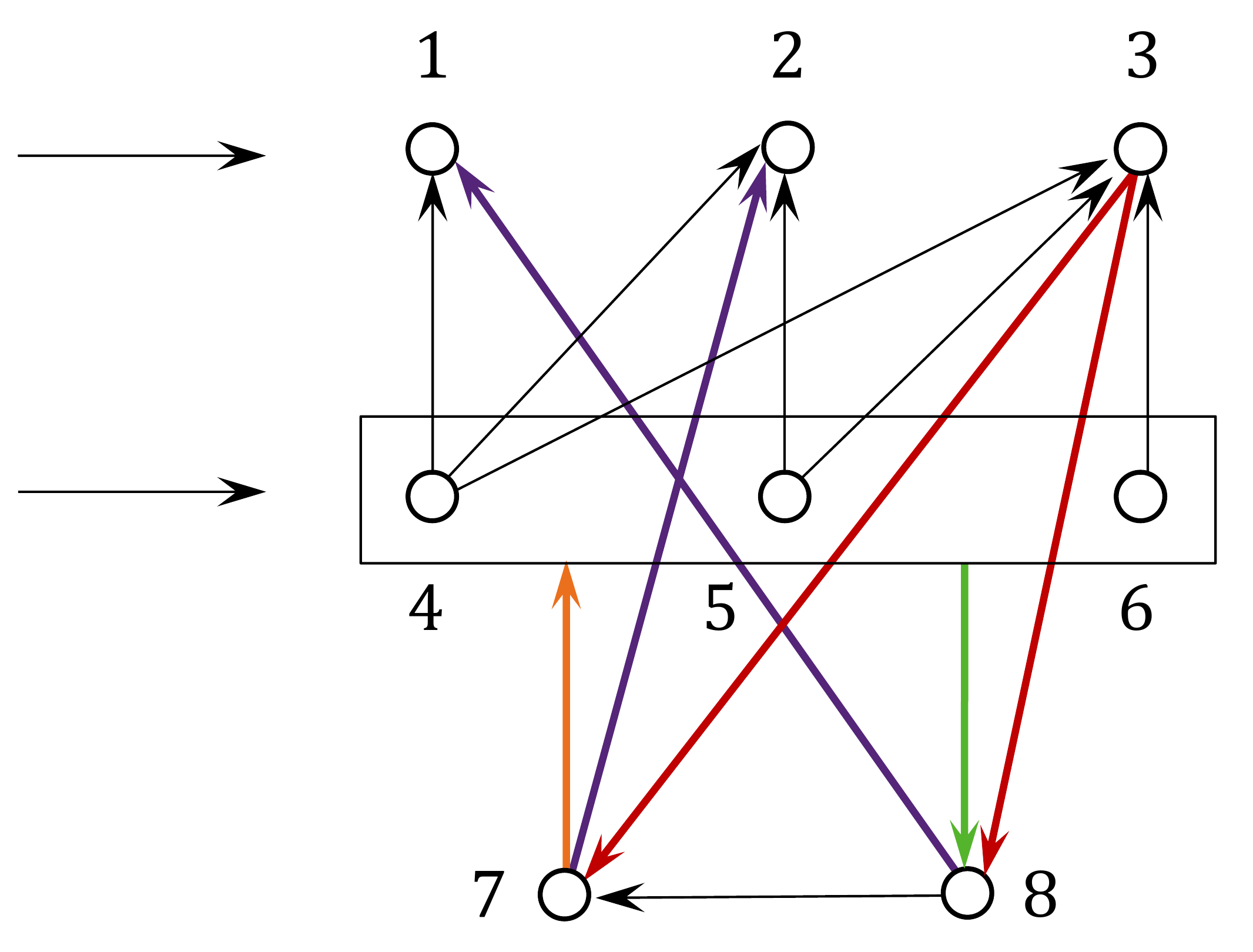}}
		\caption{Graph $\overrightarrow{B_8}$.}
		\label{fig4}
	\end{figure}
	
	\begin{lemen}\phantomsection\label{lemma4}
		The graph $\overrightarrow{B_8}$ is $\overrightarrow{C _4 }$-irregular.	
	\end{lemen}
	\begin{proof}
		There are exactly $14$ quadrangles in the graph $\overrightarrow{B_8}$: 	
		$(1,2,3,8)$, $(1,3,7,4)$, $(2,3,7,4)$, $(2,3,7,5)$, $(2,3,8,7)$, $(3,7,4,5)$, $(3,7,4,6)$, 	
		$(3,7,5,6)$, $(3,8,7,4)$,  $(3,8,7,5)$, $(3,8,7,6)$, $(4,5,8,7)$, $(4,6,8,7)$, $(5,6,8,7).$
		Next, from Table~\ref{dtabl2} we obtain that the $\overrightarrow{C _4 }$-deg\-rees of the vertices in $\overrightarrow{B_8}$ are pairwise distinct. Therefore, $\overrightarrow{B_8}$ is a $\overrightarrow{C _4 }$-irregular graph.
		\begin{table} [ht]\small
			\caption{$\overrightarrow{C _4 }$-degrees of vertices in graph $\overrightarrow{B_8}$}\label{dtabl2}
			\centerline{\begin{tabular}{ | m{4.5em} | m{0.4cm}| m{0.4cm} |m{0.4cm} |m{0.4cm} |m{0.4cm} |m{0.4cm} |m{0.4cm} | m{0.4cm} | }
					\hline
					vertex & 1 & 2 & 3 & 4 & 5 & 6 & 7 & 8 \\
					\hline
					$\overrightarrow{C_4}$-degree & 2 & 4 & 11 & 7 & 6 & 5 & 13 & 8 \\
					\hline
			\end{tabular}}
		\end{table}	
	\end{proof}
	
	\subsection{\texorpdfstring{$\overrightarrow{C _4 }$}{C_4}-irregular graph of order 10}
	
	\begin{dfn}We define the graph $\overrightarrow{B_{10}}$ (see Fig.~\ref{fig5}) of order $10$ by:
		$
		V(\overrightarrow{B_{10}})=\{1,2,\ldots,10\},$
		\begin{align*}A(\overrightarrow{B_{10}}) &=\{(1,2), (1,3), (1, 4), (2,3), (2,4), (3,4), (3,10), (4, 9), (5, 1), (5, 2),(5, 3),(5, 4)\}\\
			&\phantom{=}\cup \{(5,6), (5,7), (5,8), (5,10), (6,2), (6,3), (6, 4), (6, 7),(6, 8), (6,10), (7, 3)\}\\
			&\phantom{=}\cup  \{(7, 4),  (7, 8),  (7, 10),(8, 4), (8,10),  (9, 5),(9,6), (9,7),(9,8),(10,4),(10,9)\}.\end{align*}\end{dfn}
	\begin{figure}[h!]\small
		\centerline{\includegraphics[width=0.5\textwidth]{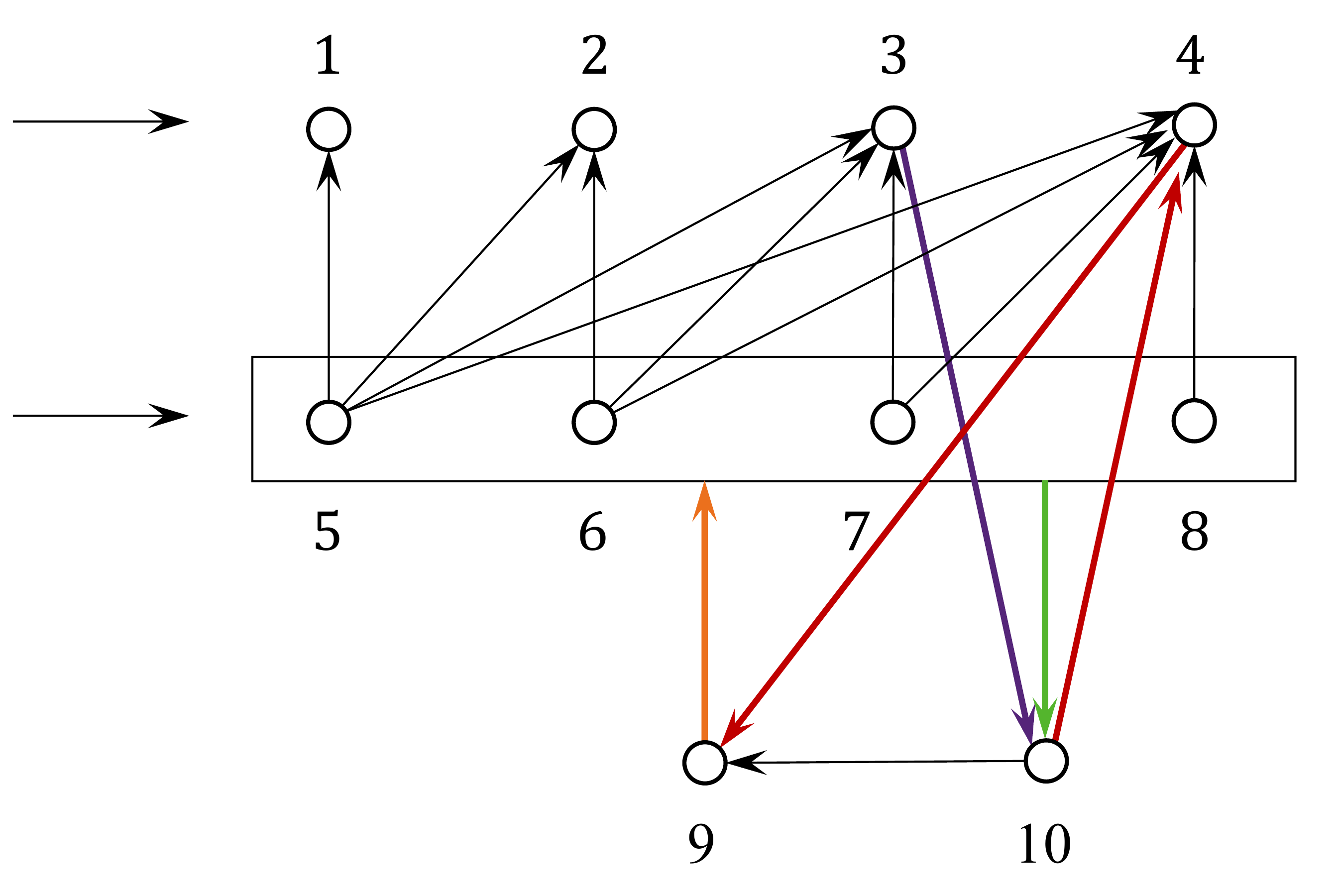}}	
		\caption{Graph $\overrightarrow{B_{10}}$.}
		\label{fig5}
	\end{figure}
	\begin{lemen}\phantomsection\label{lemma5}
		The graph $\overrightarrow{B_{10}}$ is $\overrightarrow{C _4 }$-irregular.
	\end{lemen}
	\begin{proof}
		Let's list all the quadrangles of the graph $\overrightarrow{B_{10}}$:
		$(1, 4, 9, 5)$, $(2, 4, 9, 5)$, $(2, 4, 9, 6)$, $(3, 4, 9, 5)$, $(3, 4, 9, 6)$, $(3, 4, 9, 7)$, $(3, 10, 9, 5)$, $(3, 10, 9, 6)$, $(3, 10, 9, 7)$,
		$(4, 9, 5, 6)$, $(4, 9, 5, 7)$, $(4, 9, 5, 8)$, $(4, 9, 5, 10)$, $(4, 9, 6, 7)$, $(4, 9, 6, 8)$, $(4, 9, 6, 10)$, $(4, 9, 7, 8)$, $(4, 9, 7, 10)$, $(4, 9, 8, 10)$, $(5, 6, 10, 9)$, $(5, 7, 10, 9)$,
		$(5, 8, 10, 9)$, $(6, 7, 10, 9)$, $(6, 8, 10, 9)$, $(7, 8, 10, 9)$. 
		
		From Table~\ref{dtabl3} it follows that $\overrightarrow{B_{10}}$ is a $\overrightarrow{C _4 }$-irregular graph.
		\begin{table} [ht]\small
			\caption{$\overrightarrow{C _4 }$-degrees of vertices in graph $\overrightarrow{B_{10}}$}\label{dtabl3}
			\centerline{\begin{tabular}{ | m{4.5em} | m{0.4cm}| m{0.4cm} |m{0.4cm} |m{0.4cm} |m{0.4cm} |m{0.4cm} |m{0.4cm} | m{0.4cm} | m{0.4cm} | m{0.4cm} |  }
					\hline
					vertex & 1 & 2 & 3 & 4 & 5 & 6 & 7 & 8 & 9 & 10 \\
					\hline
					$\overrightarrow{C_4}$-degree & 1 & 2 & 6 & 16 & 11 & 10 & 9 & 7 & 25 & 13 \\
					\hline
			\end{tabular}}
		\end{table}	
	\end{proof}
	
	\subsection{\texorpdfstring{$\overrightarrow{C_4}$}{C_4}-irregular graphs of even order greater than 10}
	In this subsection, we present a $\overrightarrow{C_4}$-irregular graph of order $k$ for every even $k \ge 12$.
	\smallskip
	
	\begin{dfn} Let $l \ge 5$ be an integer. Consider the graph $\overrightarrow{B_{2l+2}}$ (see Fig.~\ref{fig6}) defined by:
		\begin{align*}
			V(\overrightarrow{B_{2l+2}}) &= V_1 \cup V_2 \cup \{2l+1,2l+2\}, \ V_1=\{1,2,\dots,l\}, \ V_2=\{l+1,l+2,\dots,2l\}; \\[1ex]
			A(\overrightarrow{B_{2l+2}}) &= \{(i,j) \mid i,j \in V_1,i<j\} \cup \{(i,j) \mid i,j \in V_2,i<j\} \cup \{(2l+1,i) \mid i \in V_2\} \\[1ex]
			&\phantom{=} \cup \{(i,2l+2) \mid i \in V_2\} \cup \{(i,j) \mid i \in V_2,j \in V_1,i-j \leq l\} \\[1ex]
			&\phantom{=} \cup \{(l,2l+1),(2l+2,l),(2l+2,2l+1)\}.
	\end{align*}\end{dfn}
	\begin{figure}[h!]
		\centerline{\includegraphics[width=0.5\textwidth]{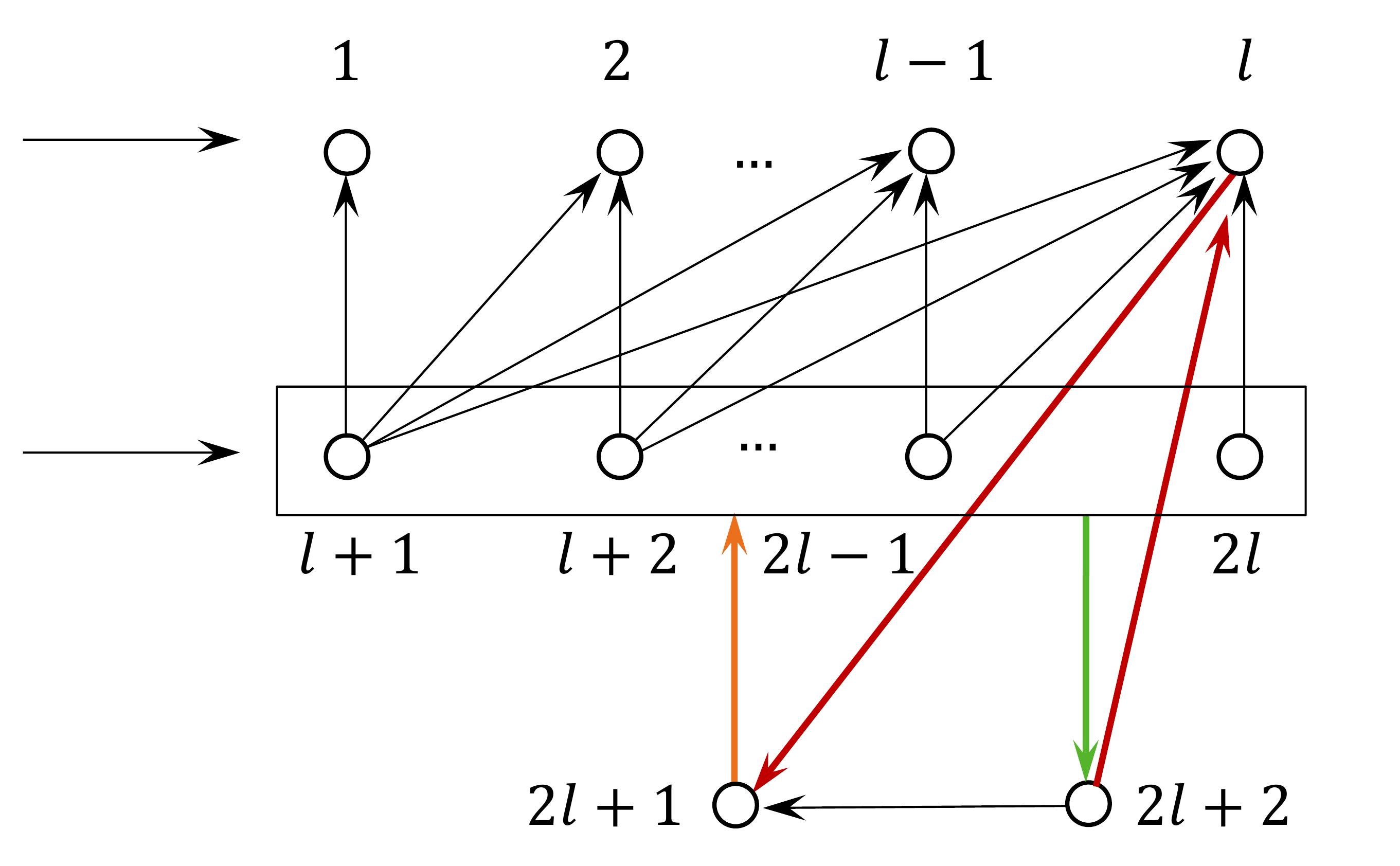}}	
		\caption{Graph $\overrightarrow{B_{2l+2}}$.}
		\label{fig6}
	\end{figure}
	\begin{dfn} Let $b_i = \overrightarrow{C_4}\deg_{\overrightarrow{B_{2l+2}}}(i)$ for all $i \in V(\overrightarrow{B_{2l+2}})$.\end{dfn}
	\begin{lemen}\phantomsection\label{lemma6}
		For each integer $l \ge 5$,  $\overrightarrow{C_4}$-degrees of vertices in   $\overrightarrow{B_{2l+2}}$ are equal to
		
		\smallskip
		$1) \ b_i=i\ \  \forall i \in \{1,2,\ldots,l-1\};$
		\smallskip	
		
		$2) \ b_l=l^2;$ \smallskip	
		
		$3) \  b_i= 4l-i-1\ \  \forall i \in \{l+1,l+2,\ldots,2l\};$
		\smallskip	
		
		$4) \ b_{2l+1}=l(3l-1)/2;$
		\smallskip
		
		$5) \ b_{2l+2}=l(l+1)/2.$
	\end{lemen}
	\begin{proof}
		We denote by $x,y,z$ the numbers of quadrangles in the graph $\overrightarrow{B_{2l+2}}$ that contain the arcs $(2l+2,l)$, $(2l+2,2l+1)$, and $(l,2l+1)$, respectively.
		
		Calculate $x,y,z$. First of all, we note that every quadrangle with the arc~$(2l+2,l)$ contains the arc $(l,2l+1)$ and has the form
		$(l,2l+1,u,2l+2)$, where $u\in \{l+1,l+2,\ldots,2l\}$. Thus, $x=l$.
		
		Furthermore, the set of all quadrangles with the arc $(2l+2,2l+1)$ is $\{(u,v,2l+2,2l+1)\}$, where $u,v\in \{l+1,l+2,\ldots,2l\}, u<v$.  It follows from this that $y=C_{l}^{2}= l(l-1)/2.$
		
		To find $z$, we partition the quadrangles  with the arc $(l, 2l+1)$ into three groups. The first group consists of $l$ quadrangles of the form $(l, 2l+1, u, 2l+2)$, where $u \in \{l+1, l+2,\dots, 2l\}$. The second group contains exactly $C_{l}^{2}= l(l-1)/2$ quadrangles of the form $(l, 2l+1, u, v)$, where $u, v \in \{l+1,l+2, \dots, 2l\}$, $u < v$. The third group contains quadrangles of the form $(u, l, 2l+1, v)$ with $u \in \{1, 2, \dots, l-1\}$, $v \in \{l+1, l+2, \dots, l+u\}$, and their number equals $\sum_{u=1}^{l-1} u = l(l-1)/2$. 
		Thus, $z =l + l(l-1)/2 + l(l-1)/2 = l^2$.
		
		Let us find $b_l$. It is clear that every quadrangle in  $\overrightarrow{B_{2l+2}}$ that contains vertex $l$ also has the arc $(l,2l+1)$ and vice versa. Therefore,
		$b_l=z =l^2.$
		
		Next we find $b_{2l+1}$. From the construction of the graph $\overrightarrow{B_{2l+2}}$, it follows that 
		\[
		b_{2l+1}=y+z =\frac {l(l-1)}{2}+l^2=\frac {l(3l-1)}{2}.
		\]
		
		Now let's calculate $b_{2l+2}$. It is easy to see that
		$$b_{2l+2}=x+y =l+\frac {l(l-1)}{2}=\frac {l(l+1)}{2}.$$
		
		Let us fix $i \in \{1,2,\ldots,l-1\}$. It is obvious that all quadrangles in $\overrightarrow{B_{2l+2}}$ with vertex~$i$ have the form $(i,l,2l+1,j)$, where $j \in \{l+1,l+2,\ldots,l+i\}$. From this we conclude that	
		$$b_i=i \quad \forall i \in \{1,2,\ldots,l-1\}.$$
		
		To complete the proof of Lemma~\ref{lemma6}, consider $i \in \{l+1,l+2,\ldots,2l\}$. In  $\overrightarrow{B_{2l+2}}$, one can distinguish exactly six types of quadrangles with vertex $i$.
		\begin{enumerate}[label=Type~\arabic*:, leftmargin=17mm]
			
			\item $(j,l,2l+1,i)$, $j \in \{i-l,i-l+1,\ldots,l-1\}$. Here and below we assume that the set $\{a, a+1,\ldots,b\}$ of consecutive vertices $a, a+1,\ldots,b$ in $\overrightarrow{ B_{2l+2}}$ is empty if $a > b$. There are $2l-i$  quadrangles of type 1.
			\smallskip
			
			\item  $(l,2l+1,j,i)$,
			$j \in \{l+1,l+2,\ldots,i-1\}$.  $\overrightarrow{B_{2l+2}}$ contains $i-l-1$ quadrangles of type 2.
			\smallskip
			
			\item  $(l,2l+1,i,j)$,
			$j \in \{i+1,i+2,\ldots,2l\}$.  We have $2l-i$ such quadrangles.
			\smallskip
			
			\item $(i,j,2l+2,2l+1)$,
			$j \in \{i+1,i+2,\ldots,2l\}$.  There are  $2l-i$ of them.
			\smallskip
			
			\item $(j,i,2l+2,2l+1)$,
			$j \in \{l+1,l+2,\ldots,i-1\}$.  We have  $i-l-1$ quadrangles of type 5.
			\smallskip
			
			\item  $(l,2l+1,i,2l+2)$. There is exactly one quadrangle of this type.\end{enumerate}
		
		Summing these contributions, we obtain:
		\[
		b_i = 3(2l-i) + 2(i-l-1) + 1 = 4l-i-1 \quad \forall i \in \{l+1, l+2, \dots, 2l\}. \qedhere
		\]
	\end{proof}
	\begin{theorem}\phantomsection\label{t2}
		For every integer $l \ge 5$, the graph $\overrightarrow{B_{2l+2}}$ is $\overrightarrow{C _4 }$-irregular.
	\end{theorem}
	\begin{proof}
		Let $l \ge 5$ be an integer. Let's consider the graph $\overrightarrow{B_{2l+2}}$. By Lemma~\ref{lemma6}, we infer that the $\overrightarrow{C_4 }$-degrees of the vertices in   $\overrightarrow{B_{2l+2}}$  are pairwise distinct: 	
		$$b_1<b_2<\ldots<b_{l-1}=l-1<2l-1=b_{2l}<b_{2l-1}<b_{2l-2}<\ldots< b_{l+1},$$		
		$$b_{l+1}=3l-2<\frac {l(l+1)}{2}=b_{2l+2}<b_l=l^2  <\frac{l(3l-1)}{2}=b_{2l+1}.$$	    	
		Hence, the graph $\overrightarrow{B_{2l+2}}$ is $\overrightarrow{C _4 }$-irregular.  
	\end{proof}
	\begin{cor}\phantomsection
		\label{cor3}
		The number of $\overrightarrow{C_4}$-irregular oriented graphs is infinite.
	\end{cor}
	
	\subsection{Criterion for the existence of a \texorpdfstring{$\overrightarrow{C _4 }$}{C_4}-irregular graph of order $k$}
	
	\begin{theorem}\phantomsection\label{t3}
		There exists a nontrivial $\overrightarrow{C_4}$-irregular graph of order $k$ if and only if $k$ is an integer and $k \ge 7$.
	\end{theorem}
	\begin{proof}
		\emph {Sufficiency.} For $k=7$ and any even $k \ge 8$, by Lemmas \ref{lemma3}--\ref{lemma5} and Theorem~\ref{t2}, there exists a $\overrightarrow{C_4}$-irregular graph $\overrightarrow{B_k }$ of order $k$. It is easy to see that all vertices of such graphs have positive $\overrightarrow{C_4}$-degrees. Therefore, a $\overrightarrow{C_4}$-irregular graph of odd order $k \ge 9$ can be constructed by adding an isolated vertex to the graph $\overrightarrow{B_{k-1}}$ . The sufficiency is proven.
		
		\emph {Necessity.}
		Let us assume the opposite, that there exists a non-trivial $\overrightarrow{C_4}$-irregular graph $H$ of order $k < 7$.  Clearly, $k \in \{4,5,6\}$.
		It is easily verified that any graph of order $4$ contains at most one quadrangle. Therefore, $k \ne 4$.
		
		If $k = 5$, then for each $v \in V(H)$, there are exactly four $4$-element subsets of $V(H)$ containing~$v$. Since the vertices of each such subset belong to at most one quadrangle in $H$, we have \mbox{$\overrightarrow{C_4}\deg_H(v) \leq 4$}. Therefore, in the $\overrightarrow{C_4}$-irregular graph $H$ of order $5$, there is a vertex $u$ with $\overrightarrow{C_4}\deg_H(u)=0$.
		Given that at most one quadrangle can be formed from the vertices of the set  $V(H)\setminus \{u\}$, we obtain that the $\overrightarrow{C_4}$-degrees of all vertices in $H$ are at most $1$. Contradiction.
		
		Finally, if $k=6$, consider two vertices, $x$ and $y$, in the graph $H$, with maximum and minimum $\overrightarrow{C_4}$-degrees $a$ and $b$, respectively. It is clear that $a \ge b+5$. Next, remove vertex $y$ from $H$, and let the resulting graph of order $5$ be denoted by $H_1$. In this case, the total number of quadrangles decreases by exactly $b$. Therefore, $\overrightarrow{C_4}\deg_{H_1}(x) \ge a-b \ge 5$. However, as previously noted, the $\overrightarrow{C_4}$-degree of a vertex in a graph of order $5$ cannot exceed $4$. This yields a contradiction.
		
		From the foregoing, it follows that $k$ is an integer and $k \ge 7$. The necessity has been proven. 
	\end{proof}
	
	\section{On \texorpdfstring{$\overrightarrow{C _3 }$}{C_3}-irregular graphs}\label{dsec4}
	We continue our investigation by examining $\overrightarrow{C_3}$-irregular graphs. We address the question of what values the order of a $\overrightarrow{C_3}$-irregular graph can have.
	\begin{dfn}
		In this section, every graph isomorphic to $\overrightarrow{C _3}$ is called a triangle.\end{dfn}
	
	\subsection{\texorpdfstring{$\overrightarrow{C _3 }$}{C_3}-irregular graph of order 10}
	
	\begin{dfn} Consider the graph $\overrightarrow{D_{10}}$ with the set of vertices $V(\overrightarrow{D_{10}})=\{1,2,\ldots,10\}$ and the set of arcs  	
		\begin{align*}
			A(\overrightarrow{D_{10}})&=\{(1,3),(1,4),(1,9),(2,3),(2,4),(2,9),(3,8),(3,9),(3,10),(4,3),(4,9),(4,10)\}\\
			&\phantom{=} \cup \{(5,1),(5,2),(5,3),(5,4),(5,6),(6,2),(6,3),(6,4),(6,7),(7,3),(7,4),(7,5)\}\\
			&\phantom{=} \cup \{(8,4),(9,5),(9,6),(9,7),(9,8),(9,10),(10,1),(10,2),(10,5),(10,6)\}.
		\end{align*}
	\end{dfn}
	
	We depict graph $\overrightarrow{D_{10}}$ in Fig.~\ref{fig7}. In Section~\ref{dsec4}, we also use the notation from Remark~\ref{rem1}.
	\begin{figure}[h!]\small
		\centerline{\includegraphics[width=0.41\textwidth]{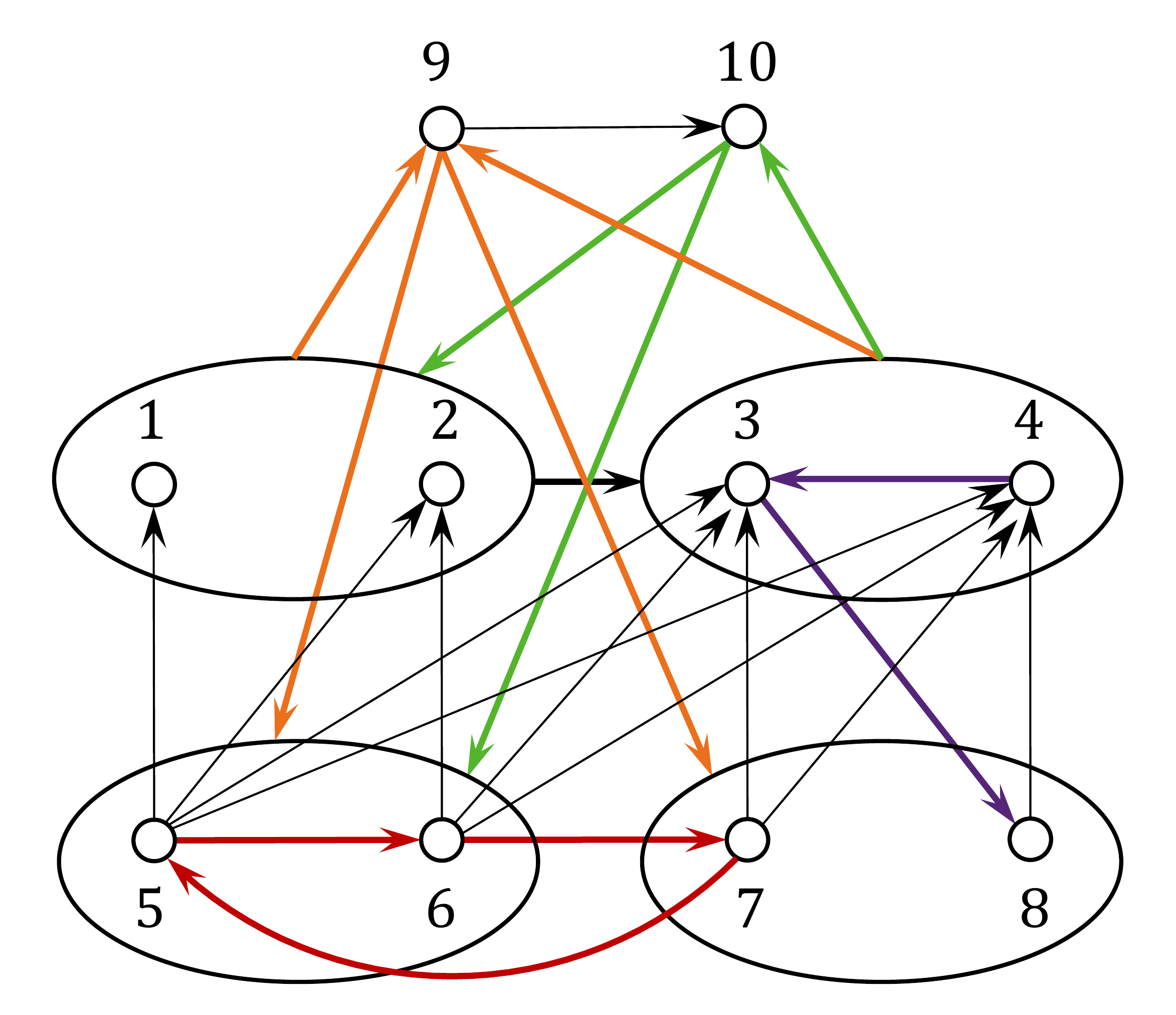}}	
		\caption{Graph $\overrightarrow{D_{10}}$.}
		\label{fig7}
	\end{figure}	
	\begin{lemen}\phantomsection\label{lemma7}
		The graph $\overrightarrow{D_{10}}$ is $\overrightarrow{C _3 }$-irregular.
	\end{lemen}
	\begin{proof}
		Let's list all the triangles in  $\overrightarrow{D_{10}}$: $(1, 3, 10)$, $(1, 4, 10)$, $(1, 9, 5)$, $(1, 9, 10)$, $(2, 3, 10)$, $(2, 4, 10)$, $(2, 9, 5)$, $(2, 9, 6)$, $(2, 9, 10)$, $(3, 8, 4)$, $(3, 9, 5)$, $(3, 9, 6)$, $(3, 9, 7)$, $(3, 10, 5)$, $(3, 10, 6)$, $(4, 9, 5)$, $(4, 9, 6)$, $(4, 9, 7)$, $(4, 9, 8)$,
		$(4, 10, 5)$, $(4, 10, 6)$, $(5, 6, 7)$.
		
		Based on Table~\ref{dtabl4}, we conclude that  $\overrightarrow{D_{10}}$ is a $\overrightarrow{C _3 }$-irregular graph.
		\begin{table} [ht]\small
			\caption{$\overrightarrow{C _3 }$-degrees of vertices in graph $\overrightarrow{D_{10}}$}\label{dtabl4}
			
			\	
			\centerline{\begin{tabular}{ | m{4.5em} | m{0.4cm}| m{0.4cm} |m{0.4cm} |m{0.4cm} |m{0.4cm} |m{0.4cm} |m{0.4cm} | m{0.4cm} | m{0.4cm} | m{0.4cm} |   }
					\hline
					vertex & 1 & 2 & 3 & 4 & 5 & 6 & 7 & 8 & 9 & 10 \\
					\hline
					$\overrightarrow{C_3}$-degree & 4 & 5 & 8 & 9 & 7 & 6 & 3 & 2 & 12 & 10 \\
					\hline
			\end{tabular}}
		\end{table}	
	\end{proof}
	
	\subsection{\texorpdfstring{$\overrightarrow{C _3 }$}{C_3}-irregular graph of order 12}
	
	\begin{dfn} We define the graph $\overrightarrow{D_{12}}$ (see Fig.~~\ref{fig8}) as follows: $V(\overrightarrow{D_{12}})=\{1,2,\ldots,12\}$,	
		\begin{align*}
			A(\overrightarrow{D_{12}}) &= \{(1,3),(1,4),(1,9),(2,3),(2,4),(2,9),(3,4),(3,9),(3,10),(4,9),(4,10)\}\\
			&\phantom{=} \cup \{(4,12),(5,1),(5,2),(5,3),(5,4),(6,2),(6,3),(6,4),(7,3),(7,4),(8,4)\}\\
			&\phantom{=} \cup \{(9,5),(9,6),(9,7),(9,8),(9,11),(10,1),(10,2),(10,5),(10,6),(10,9)\}\\
			&\phantom{=} \cup \{(10,12),(11,10),(12,1),(12,2),(12,3),(12,5),(12,6),(12,7),(12,8)\}.
	\end{align*} \end{dfn}
\vspace{-1ex}
	\begin{figure}[h!]\centering\includegraphics[width=0.41\textwidth]{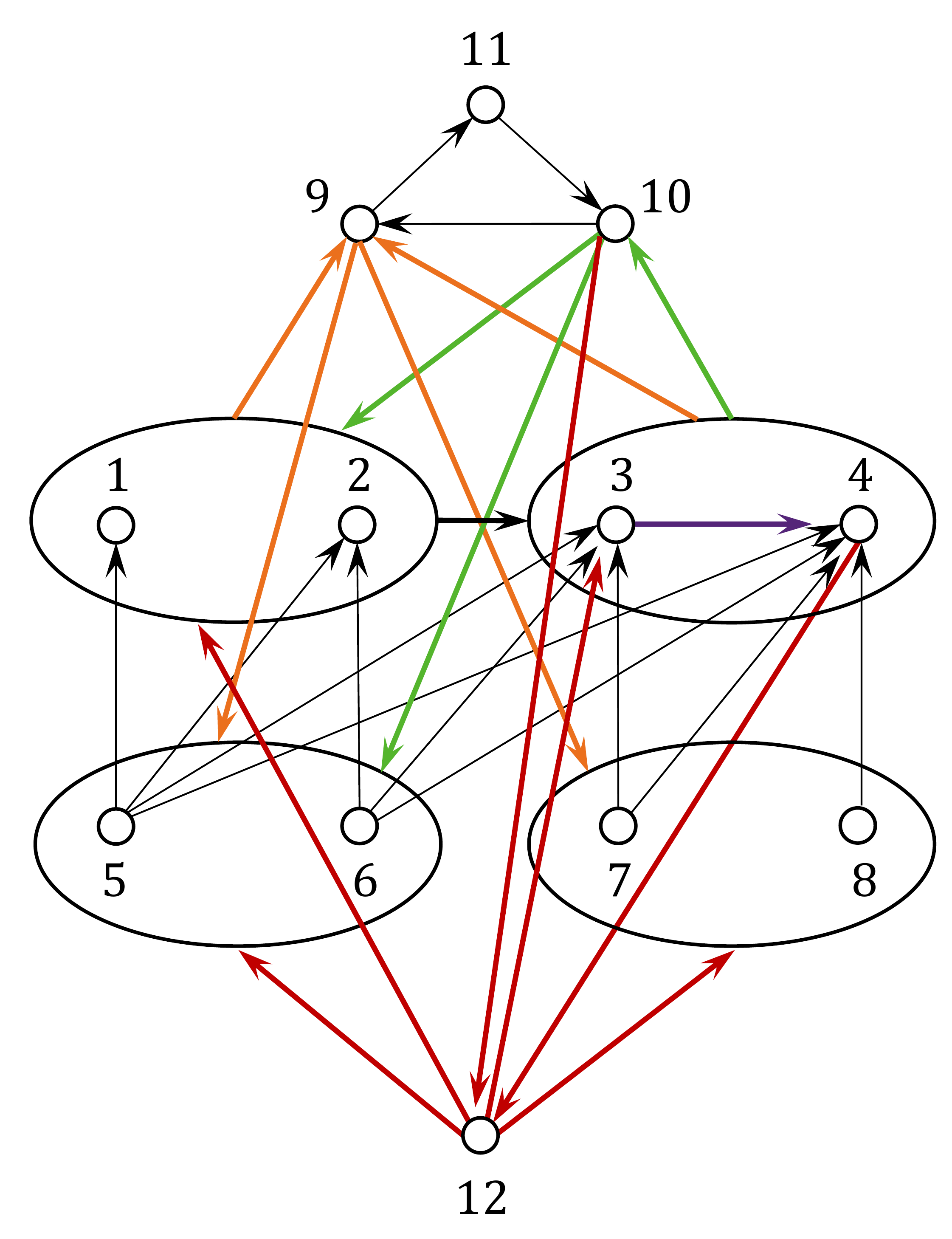}	
		\caption{Graph $\overrightarrow{D_{12}}$.}
		\label{fig8}
	\end{figure}
	\begin{lemen}\phantomsection\label{lemma8}
		The graph $\overrightarrow{D_{12}}$ is $\overrightarrow{C _3 }$-irregular.
	\end{lemen}
	\begin{proof}
		It is not difficult to verify that 	 $\overrightarrow{D_{12}}$ contains exactly $27$ triangles:
		$(1, 3, 10)$, $(1, 4, 10)$, $(1, 4, 12)$, $(1, 9, 5)$, $(2, 3, 10)$, $(2, 4, 10)$,
		$(2, 4, 12)$, $(2, 9, 5)$, $(2, 9, 6)$,
		$(3, 4, 12)$, $(3, 9, 5)$, $(3, 9, 6)$, $(3, 9, 7)$, $(3, 10, 5)$, $(3, 10, 6)$,
		$(3, 10, 12)$, $(4, 9, 5)$, $(4, 9, 6)$,
		$(4, 9, 7)$, $(4, 9, 8)$,  $(4, 10, 5)$,
		$(4, 10, 6)$, $(4, 12, 5)$,
		$(4,12,6)$, $(4,12,7)$, $(4,12,8)$,  $(9,11,10).$
		From Table~\ref{dtabl5} it follows that $\overrightarrow{D_{12}}$ is a $\overrightarrow{C _3 }$-irregular graph.
		\begin{table} [ht]\small
			\caption{$\overrightarrow{C _3 }$-degrees of vertices in graph $\overrightarrow{D_{12}}$}\label{dtabl5}
			
			\	
			
			\centerline{\begin{tabular}{ | m{4.5em} | m{0.4cm}| m{0.4cm} |m{0.4cm} |m{0.4cm} |m{0.4cm} |m{0.4cm} |m{0.4cm} | m{0.4cm} | m{0.4cm} | m{0.4cm} | m{0.4cm} | m{0.4cm} |  }
					\hline
					vertex & 1 & 2 & 3 & 4 & 5 & 6 & 7 & 8 & 9 & 10 & 11 & 12 \\
					\hline		
					$\overrightarrow{C_3}$-degree & 4 & 5 & 9 & 15 & 7 & 6 & 3 & 2 & 11 & 10 & 1 & 8 \\
					\hline
			\end{tabular}}
		\end{table}		
	\end{proof}
	
	\subsection{\texorpdfstring{$\overrightarrow{C _3 }$}{C_3}-irregular graphs of order $k \ge 14$,  $k \equiv 2 \, (\mathrm{mod} \ 4)$}
	In this subsection, for every integer $k \ge 14$ with $k \equiv 2 \pmod 4$, we construct a $\overrightarrow{C_3}$-irregular graph of order $k$.
	\begin{dfn} For an arbitrary integer $m \ge 3$, consider the graph $\overrightarrow{D_{4m+2}}$, which is shown in Fig.~\ref{fig9} and given by the following vertex and arc sets:
		\begin{align*}
			V(\overrightarrow{D_{4m+2}}) &= V_1\cup V_2\cup V_3\cup V_4\cup \{4m+1,4m+2\}, \text{ where} \\[1ex]
			&\begin{aligned}
				V_1&=\{1,2,\dots,m\}, & V_2&=\{m+1,m+2,\dots,2m\}, \\[0.8ex]
				V_3&=\{2m+1,2m+2,\dots,3m\}, \quad & V_4&=\{3m+1,3m+2,\dots,4m\};
			\end{aligned} \\[1.5ex]
			A(\overrightarrow{D_{4m+2}}) &= \{(i,j) \mid i \in V_1,j \in V_2\} \cup \{(i,j) \mid i-j \leq 2m,i \in (V_3 \cup V_4),j \in (V_1 \cup V_2)\} \\[1.2ex]
			&\phantom{=} \cup \{(i,4m+1) \mid i \in (V_1 \cup V_2)\} \cup \{(4m+1,i) \mid i \in (V_3 \cup V_4)\} \\[1.2ex]
			&\phantom{=} \cup \{(i,4m+2) \mid i \in V_2\} \cup \{(4m+2,i) \mid i \in (V_1 \cup V_3)\}.
	\end{align*}\end{dfn}
	
	\begin{figure}[h!]
		\centering\includegraphics[width=0.55\textwidth]{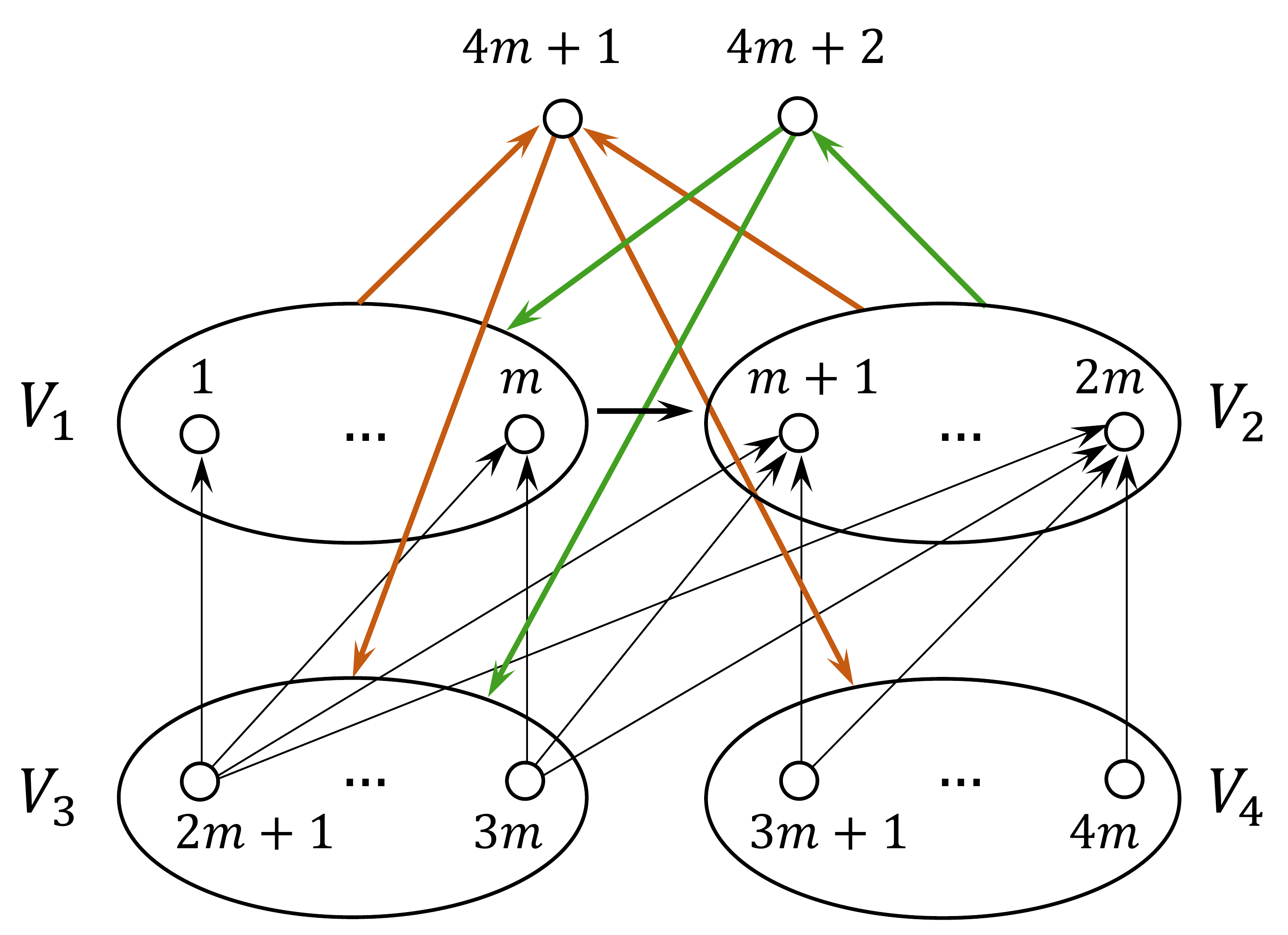}	
		\caption{Graph $\overrightarrow{D_{4m+2}}$.}
		\label{fig9}
	\end{figure}
	\begin{dfn} Let $d_i = \overrightarrow{C_3}\deg_{\overrightarrow{D_{4m+2}}}(i)$ for each $i \in V(\overrightarrow{D_{4m+2}})$.\end{dfn}
	
	\begin{lemen}\phantomsection\label{lemma9}
		For every integer  $m \ge 3$, the  $\overrightarrow{C_4}$-degrees of the vertices in  $\overrightarrow{D_{4m+2}}$ are equal to
		
		\smallskip
		$1) \ d_{i}= m+i\ \  \forall i \in V_1;$
		
		\smallskip	
		$2) \ d_{i}= 2m+i\ \  \forall i \in V_2;$
		
		\smallskip	
		$3) \ d_{i}=5m-i+1\ \  \forall i \in V_3;$
		
		\smallskip	
		$4) \ d_{i}=4m-i+1\ \  \forall i \in V_4;$
		
		\smallskip	
		$5) \ d_{4m+1}=2m^2+m;$
		
		\smallskip
		$6) \ d_{4m+2}=2m^2.$
	\end{lemen}
	\begin{proof}
		Let's fix an integer $m \ge 3$ and consider the graph $\overrightarrow{D_{4m+2}}$.
		
		First of all, note that each vertex $i \in V_1$ belongs to exactly two types of triangles in $\overrightarrow{D_{4m+2}}$: $(i, j, 4m+2)$, where $j \in V_2$, and $(i, 4m+1, j)$, where $j \in \{2m+1, 2m+2, \dots, 2m+i\}$. Thus,
		\[
		d_i = |V_2| + i = m + i \quad \forall i \in V_1.
		\]
		
		Next, let $i \in V_2$. Observe that $\overrightarrow{D_{4m+2}}$ contains exactly two types of triangles with vertex $i$: $(i, 4m+2, j)$, where $j \in (V_1 \cup V_3)$, and $(i, 4m+1, j)$, where $j \in \{2m+1, 2m+2, \dots, 2m+i\}$. Hence,
		\[
		d_i = |V_1 \cup V_3| + i = 2m + i \quad \forall i \in V_2.
		\]
		
		If $i \in V_3$, there are exactly two types of triangles in $\overrightarrow{D_{4m+2}}$ containing vertex $i$: $(i,j,4m+1)$, where $j \in \{i-2m,i-2m+1, \dots, 2m\}$, and $(i,j,2m+2)$, where $j \in V_2$. This yields
		\[
		d_i = 4m - i + 1 + |V_2| = 4m - i + 1 + m = 5m - i + 1 \quad \forall i \in V_3.
		\]
		
		Now consider $i \in V_4$. Since vertex $i$ in the graph $\overrightarrow{D_{4m+2}}$ belongs only to triangles of the form $(i, j, 4m+1)$, where $j \in \{i-2m, i-2m+1, \dots, 2m\}$, we have
		\[
		d_i = 4m - i + 1 \quad \forall i \in V_4.
		\]	
		
		Let's find $d_{4m+2}$. Since all triangles in the graph $\overrightarrow{D_{4m+2}}$ containing vertex $4m+2$ are of the form $(u,v,4m+2)$, where $u \in V_1 \cup V_3$ and $v \in V_2$, it follows that
		\[
		d_{4m+2}=|V_1 \cup V_3| \cdot |V_2|=2m \cdot m=2m^2.
		\]
		
		To find $d_{4m+1}$, we note that in the graph $\overrightarrow{D_{4m+2}}$, all triangles of the form $(u,4m+1,v)$, where
		$u \in (V_1 \cup V_2)$,  $v \in \{2m+1,2m+2,\ldots,2m+u\}$, and only these, contain vertex $4m+1$. Thus,  
		\[
		d_{4m+1} = \sum_{u=1}^{2m} u = 2m^2 + m. \qedhere
		\]	
	\end{proof}
	\begin{theorem}\phantomsection\label{t4}
		For every integer $m \ge 3$, the graph $\overrightarrow{D_{4m+2}}$ is $\overrightarrow{C _3 }$-irregular.
	\end{theorem}
	\begin{proof}
		Let's fix an integer $m \ge 3$ and consider the graph
		$\overrightarrow{D_{4m+2}}$.
		
		Next, for each $i \in \{1,2,3,4\}$,
		we denote by $X(V_i)$ the set of $\overrightarrow{C _3 }$-degrees of all vertices from $V_i$ in the graph $\overrightarrow{D_{4m+2}}$. From Lemma~\ref{lemma9}, it follows that 	
		\begin{gather*}
			\begin{aligned}
				X(V_4)&=\{1,2,\ldots,m\}, & X(V_1)&=\{m+1,m+2,\ldots,2m\},\\[0.5ex]
				X(V_3)&=\{2m+1,2m+2,\ldots,3m\},   &X(V_2)&=\{3m+1,3m+2,\ldots,4m\},
			\end{aligned} \\[0.5ex]
			4m < 2m^2 = d_{4m+2} < d_{4m+1} = 2m^2 + m
		\end{gather*}
		Thus, the $\overrightarrow{C _3 }$-degrees of all  vertices in  $\overrightarrow{D_{4m+2}}$ are distinct, and hence $\overrightarrow{D_{4m+2}}$  is $\overrightarrow{C _3 }$-irregular. 
	\end{proof}
	\begin{cor}\phantomsection\label{cor4}
		The number of $\overrightarrow{C_3}$-irregular oriented graphs is infinite.
	\end{cor}
	
	\subsection{\texorpdfstring{$\overrightarrow{C _3 }$}{C_3}-irregular graphs  of order $k \ge 16$, $k \equiv 0 \, (\mathrm{mod} \ 4)$ }
	We now exhibit a $\overrightarrow{C_3}$-irregular graph of order $k$ for each $k \ge 16$ divisible by $4$.
	\begin{dfn} Let $m \ge 3$ be an integer. We define the graph $\overrightarrow{D_{4m+4}}$ as follows:
		\begin{align*}
			V(\overrightarrow{D_{4m+4}}) &= V(\overrightarrow{D_{4m+2}}) \cup \{4m+3,4m+4\}, \\[1ex]
			A(\overrightarrow{D_{4m+4}}) &= A(\overrightarrow{D_{4m+2}}) \cup \{(i,2m) \mid i \in (V_2 \setminus \{2m\})\} \\[1ex]
			&\phantom{=} \cup \{(4m+3,i) \mid i \in (\{1,2,\dots,4m\} \setminus \{2m\})\} \cup \{(2m,4m+3)\} \\[1ex]
			&\phantom{=} \cup \{(2m-2,2m-1),(2m-1,4m+4),(4m+4,2m-2)\}.
	\end{align*}\end{dfn}
	
	The graph $\overrightarrow{D_{4m+4}}$ is illustrated in Fig.~\ref{fig10}. Note that in this case, we show  schematically the set of arcs
	$\{(i,j)\mid i-j\leq 2m, i \in (V_3 \cup V_4), j \in (V_1\cup V_2)\}$.
	\begin{figure}[h!]
		\centerline{
			\includegraphics[width=0.7\textwidth]{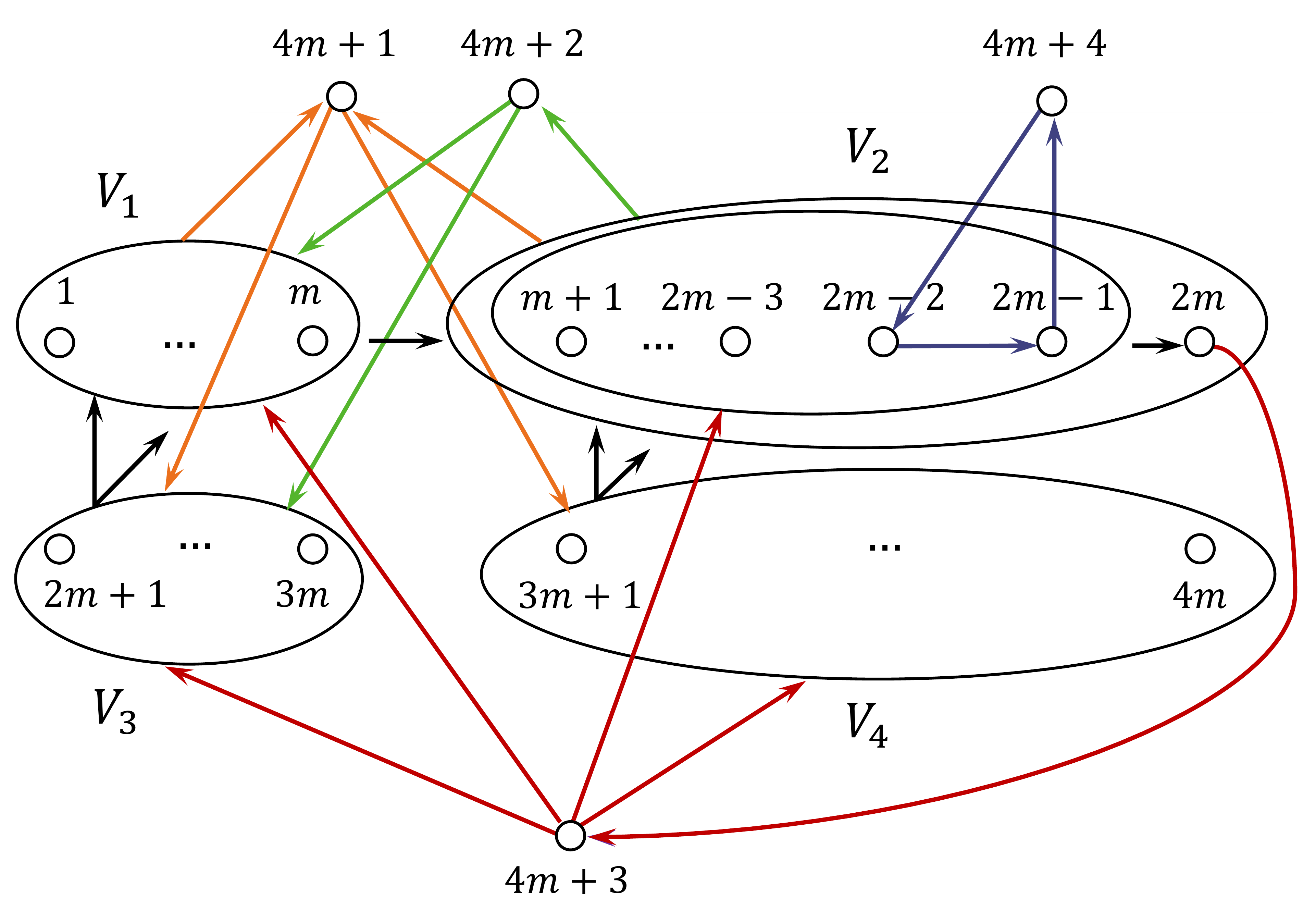}	}
		\caption{Graph $\overrightarrow{D_{4m+4}}$.}
		\label{fig10}
	\end{figure}
	
	\begin{dfn} Let $d_{i}^* = \overrightarrow{C_3}\deg_{\overrightarrow{D_{4m+4}}}(i)$ for all $i \in V(\overrightarrow{D_{4m+4}})$.\end{dfn}
	\begin{lemen}\phantomsection\label{lemma10}
		Let $m \ge 3$ be an integer. For $\overrightarrow{C_3}$-degrees of the vertices in   $\overrightarrow{D_{4m+4}}$, the following equalities are hold:
		
		\smallskip
		$1) \ d_{i}^{*}= d_{i}+1 \ \ \forall i \in \{1, 2,\ldots, 4m\} \setminus \{ 2m-2, 2m-1, 2m\};$
		
		\smallskip	
		$2) \ d_{i}^{*}= d_{i}+2\ \  \forall i \in \{2m-1, 2m-2\};$
		
		\smallskip	
		$3) \ d_{2m}^{*}=d_{2m}+4m-1;$
		
		\smallskip	
		$4) \ d_{i}^{*}= d_{i}\ \  \forall i \in \{4m+1, 4m+2\};$
		
		\smallskip	
		$5) \ d_{4m+3}^{*}=4m-1;$
		
		\smallskip
		$6) \ d_{4m+4}^{*}=1.$
	\end{lemen}
	\begin{proof}	
		By construction, every triangle in $\overrightarrow{D_{4m+2}}$ is also a triangle in $\overrightarrow{D_{4m+4}}$. Furthermore, $\overrightarrow{D_{4m+4}}$ contains exactly $4m$ new triangles that do not belong to $\overrightarrow{D_{4m+2}}$: $(2m-2, 2m-1, 4m+4)$, and  $4m-1$ triangles of the form $(i, 2m, 4m+3)$, where $i \in (\{1, 2, \ldots, 4m\} \setminus  \{2m\})$. This ensures the validity of Lemma~\ref{lemma10}.
	\end{proof}
	\begin{theorem}\phantomsection\label{t5}
		Let $m \ge 3$ be an integer. The graph $\overrightarrow{D_{4m+4}}$ is $\overrightarrow{C _3 }$-irregular.
	\end{theorem}
	\begin{proof}
		For each $H \subseteq V(\overrightarrow{D_{4m+4}})$,
		we denote by $X^*(H)$ the set of $\overrightarrow{C _3 }$-degrees of all vertices from $H$ in the graph $\overrightarrow{D_{4m+4}}$. From Lemmas~\ref{lemma9}, \ref{lemma10} it follows that
		\begin{gather*}
			d_{4m+4}^*=1, \\
			\begin{aligned}
				X^*(V_4)&=\{2,3,\ldots,m+1\},\\[0.3ex]
				X^*(V_1)&=\{m+2,m+3,\ldots,2m+1\},\\[0.3ex]
				X^*(V_3)&=\{2m+2,2m+3,\ldots,3m+1\},\\[0.3ex]
				X^*(\{ m+1, m+2, \ldots , 2m-3\})&=\{3m+2,3m+3,\ldots,4m-2\}\ \  \text{for} \ \  m \ge 4,\\[0.3ex]
				d_{4m+3}^*=4m-1, & \quad d_{2m-2}^*=4m, \quad\ \,   d_{2m-1}^*=4m+1,\\[0.3ex]
				d_{2m}^*=8m-1,  & \quad  d_{4m+2}^*=2m^2, \quad  d_{4m+1}^*=2m^2+m.
			\end{aligned}
		\end{gather*}
		
		Obviously, the vertices from the set $V(\overrightarrow{D_{4m+4}})\setminus \{ 2m, 4m+1, 4m+2\}$ have different $\overrightarrow{C_3 }$-degrees in $\overrightarrow{D_{4m+4}}$. In addition,
		$$
		d_i^*>d_j^* \ \  \forall i \in \{ 2m, 4m+1, 4m+2\},\quad    j \in V(\overrightarrow{D_{4m+4}})\backslash \{ 2m, 4m+1, 4m+2\}.
		$$
		
		Next, if $m=3$, then $d_{4m+2}^*<d_{4m+1}^*<d_{2m}^*.$ If $m \ge 4$, then $d_{2m}^*<d_{4m+2}^*<d_{4m+1}^*$.	
		From the above we conclude that the graph $\overrightarrow{D_{4m+4}}$ is $\overrightarrow{C _3 }$-irregular. 		
	\end{proof}
	
	\subsection{Criterion for the existence of a \texorpdfstring{$\overrightarrow{C _3 }$}{C_3}-irregular graph of order $k$}
	
	In this subsection, we will finally answer the question concerning the order of a $\overrightarrow{C_3}$-irregular graph (Theorem~\ref{t6}), and for this, we will need two additional lemmas.
	\begin{lemen}\phantomsection\label{lemma11}
		Let $G$ be a graph of order $k$ and $v \in V(G)$. Then
		\begin{equation*}
			\overrightarrow{C_3} \deg_G(v) \le \frac{1}{4}(k-1)^2.
		\end{equation*}	
	\end{lemen}
	\begin{proof} Note that any triangle in $G$ with a vertex $v$ has a pair of arcs $(u,v), (v,w)$, where $u,w \in V(G)$. On the other hand, for any two arcs $(u,v), (v,w) \in A(G)$, there is at most one triangle in $G$ containing these arcs. Therefore 	
		\begin{equation*}
			\overrightarrow{C_3} \deg_G(v) \le \deg_G^+(v) \cdot \deg_G^-(v) \le \deg_G^+(v)(k-1-\deg_G^+(v)) \le \frac{1}{4} (k-1)^2. \qedhere	
		\end{equation*}
	\end{proof}
	\begin{dfn} 
		For any graph $G$, let $S(G)$ denote the sum of the $\overrightarrow{C_3}$-degrees of all its vertices. For every positive integer $k$, let $S_k$ be the maximum value of $S(G)$ among all graphs $G$ of order $k$.
	\end{dfn}
	\begin{lemen}\phantomsection\label{lemma12} 	
		$S_5=15, \ S_6=24.$	
	\end{lemen}
	\begin{proof}
		Figure~\ref{fig11} shows the graphs $G_1$ and $G_2$ of order $5$ and $6$, respectively, with $S(G_1)=15$,  $S(G_2)=24$. Hence, it suffices to verify that $S_5 \leq 15$ and $S_6 \leq 24$.
		\begin{figure}[h!]
			\centering\includegraphics[width=7cm]{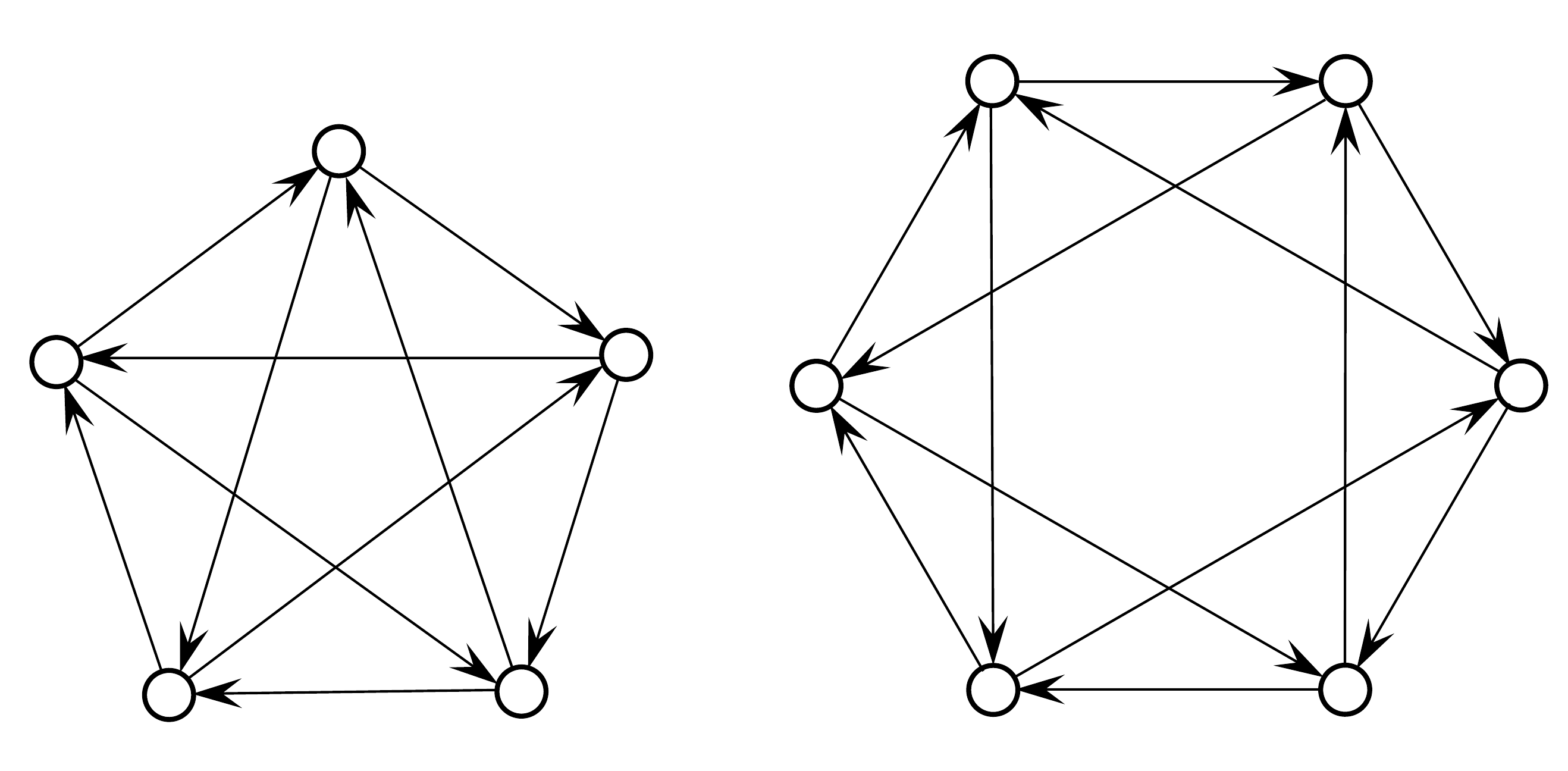}	
			\caption{Graph $G_1$ with $|G_1|=5$,
				$S(G_1)=15$ and graph $G_2$ with $|G_2|=6$,
				$S(G_2)=24$.}
			\label{fig11}
		\end{figure}		
		
		To prove $S_5 \leq 15$, assume for contradiction that there exists a graph $L$ of order $5$ with $S(L) > 15$.  Then $L$ has a vertex $v$ with $\overrightarrow{C _3} \deg_L (v) \ge 4$, otherwise  $S(L) \leq 5 \cdot 3=15$. 
		Since $\overrightarrow{C_3} \deg_L(v) \leq 4$ by Lemma~\ref{lemma11}, we have $\overrightarrow{C_3} \deg_L(v) = 4$, which implies $\deg_L^+(v) = \deg_L^-(v) = 2$. Let $(v,a), (v,b), (c,v), (d,v) \in A(L)$. Then $(a,c), (a,d), (b,c), (b,d) \in A(L)$, and the subgraph of $L$ induced by the set of vertices $\{a, b, c, d \}$ is triangle-free. Thus, $L$ contains exactly $4$ triangles, i.e., $S(L)=4 \cdot 3=12$. This contradicts the assumption $S(L) > 15$, showing that $S_5 \leq 15$.
		
		To prove $S_6 \leq 24$, assume $S(M) > 24$ for some graph $M$ of order 6. Then $M$ has a vertex $u$ with $\overrightarrow{C_3}\deg_M(u) \ge 5$, otherwise $S(M) \leq 6 \cdot 4 = 24$. Thus, $\{\deg_M^+(u), \deg_M^-(u)\} = \{2, 3\}$. Without loss of generality,  let $\deg_M^+(u)=3$ and $\deg_M^-(u)=2$. With $(u,a), (u,b), (u,c), (d,u), (e,u) \in A(M)$, assume for definiteness that $(a,d), (a,e), (b,d), (b,e), (c,d) \in A(M)$. Then all triangles in $M$, except $(a,d,u), (a,e,u), (b,d,u), (b,e,u), (c,d,u)$, belong to the set	
		\[
		A = \{(a,b,c), (a,c,b), (a,e,c), (b,e,c), (c,d,e), (c,e,u)\}.
		\]	
		
		It is clear that in each of the three pairs of triangles $\{(a,c,b), (a,e,c)\}$, $\{(a,b,c), (b,e,c)\}$, and $\{(c,d,e), (c,e,u)\}$, at most one triangle belongs to $M$. Consequently, $A$ contains at most three triangles of $M$, and so $M$ contains at most $5 + 3 = 8$ triangles in total. Thus,  $S(M) \leq 8 \cdot 3=24$. This contradicts the inequality $S(M)>24$, completing the proof for $S_6 \leq 24$. 
	\end{proof}
	\begin{theorem}\phantomsection\label{t6}
		There exists a non-trivial  $\overrightarrow{C_3}$-irregular graph of order $k$ if and only if $k$ is an integer and $k \ge 10$.
	\end{theorem}
	\begin{proof}
		\emph {Sufficiency.}
		For any even $k \geq 10$, according to Lemmas~\ref{lemma7}, \ref{lemma8}, Theorems~\ref{t4}, \ref{t5}, and their proofs, there exists a $\overrightarrow{C_3}$-irregular graph $\overrightarrow{D_k}$ of order $k$ with positive $\overrightarrow{C_3}$-degrees of vertices. Adding an isolated vertex to all such graphs yields $\overrightarrow{C_3}$-irregular graphs of any odd order, starting from $11$. Sufficiency is proven.
		
		\emph {Necessity.}
		Let's consider any non-trivial $\overrightarrow{C_3}$-irregular graph $G$ of order $k$. It is clear that $k$ is an integer and $k \ge 4$.  Suppose, by way of contradiction, that  $k \leq 9$. Let $H$ be the graph of order $m \in \{3,4,\ldots,9\}$  obtained from $G$ by removing a vertex with $\overrightarrow{C _3}$-degree equal to $0$ (if no such vertex exists, then $H=G$). By construction, the graph $H$ is $\overrightarrow{C_3}$-irregular and has positive $\overrightarrow{C _3}$-degrees of vertices, which we denote  as $a_m>a_{m-1}>\ldots>a_1>0$.
		It is easy to see that $a_m \ge m$. On the other hand, Lemma~\ref{lemma11} implies $a_m \leq (m-1)^2/4$. 
		As a consequence, $(m-1)^2/4 \ge m$, so $m \in \{6,7,8,9\}$. We now proceed to investigate all $4$ possible cases, categorized by the value~of~$m$.
		
		If $m=6$, then $a_2 \ge a_1+1, a_3\ge a_1+2,\ldots, a_6 \ge a_1+5$, and therefore $S(H) \ge 6a_1+15.$	
		Next, we consider the graph $H_1$ of order $5$, formed from $H$ by removing a vertex with $\overrightarrow{C_3}$-degree equal to~$a_1$. The following estimate contradicts $S_5 = 15$ from Lemma~\ref{lemma12}:
		\[
		S_5 \ge S(H_1) = S(H) - 3a_1 \ge 3a_1 + 15 \ge 18.
		\]
		
		If $m=7$, then
		$a_2 \ge a_1+1, a_3\ge a_1+2,\ldots, a_7\ge a_1+6$, and $S(H) \ge 7a_1+21.$	
		Now, remove from $H$ a vertex with  $\overrightarrow{C_3}$-degree equal to $a_1$. Let the resulting graph of order $6$ be denoted by  $H_2$. 
		Since $S_6 = 24$ by Lemma~\ref{lemma12}, we obtain a contradiction:
		\[
		S_6 \ge S(H_2)=S(H)-3a_1 \ge 4a_1+21 \ge 25.
		\]
		
		If $m=8$, then
		$a_2 \ge a_1+1, a_3\ge a_2+1, a_4\ge a_2+2,…, a_8 \ge a_2+6$, and therefore the inequality $S(H) \ge a_1+7a_2+21$ holds.
		Let $H_3$ be the graph of order $6$  obtained from $H$ by removing the vertices with $\overrightarrow{C _3}$-degrees $a_1, a_2$. The following estimate contradicts Lemma~\ref{lemma12}:
		\[S_6 \ge S(H_3) \ge S(H)-3a_1-3a_2 \ge  4a_2-2a_1+21 \ge 2a_1+25 \ge 27.\]	
		
		If $m=9$, then
		$a_2 \ge a_1+1, a_3 \ge a_2+1,  a_4 \ge a_3+1, a_5 \ge a_3+2,\ldots, a_9 \ge a_3+6$. It follows that  $S(H) \ge a_1+a_2+7a_3+21$. Consider the graph $H_4$ of order $6$, obtained from $H$ by removing the three vertices with $\overrightarrow{C _3}$-degrees  $a_1, a_2$, $a_3$.   As before, we reach a contradiction with the equality $S_6=24$ from Lemma~\ref{lemma12}:			
		$$S_6 \ge S(H_4) \ge S(H)-3a_1-3a_2-3a_3 \ge 4a_3-2a_2-2a_1+21 \ge
		2a_2-2a_1+25 \ge 27.$$
		
		Thus, for every positive integer $k \leq 9$, there does not exist any non-trivial $\overrightarrow{C_3}$-irregular graph of order $k$. The necessity is proven.	
	\end{proof}
	
	\section{Main result}\label{dsec5}
	From Corollaries~\ref{cor2}, \ref{cor3}, and \ref{cor4}, we obtain the main result of our work, Theorem~\ref{t7}, that confirms Conjecture~\ref{con3} for every oriented cycle $\overrightarrow{C _n}$.
	\begin{theorem}\phantomsection\label{t7}
		For any integer $n\ge3$, there exists an infinite family of $\overrightarrow{C _n}$-irregular oriented graphs.
	\end{theorem}	
	
	\phantomsection
	\pdfbookmark[1]{References}{References}
	
	\end{document}